\documentclass[11pt,reqno]{amsart}
\usepackage{lineno}

\usepackage[abs]{overpic}

\usepackage{amsmath, amscd, amssymb, amsthm,amsfonts}
\usepackage{euscript, mathrsfs, latexsym,mathabx,stmaryrd}
\usepackage{color}
\usepackage{hyperref} 

\usepackage{lmodern}
\usepackage[T1]{fontenc} 

\usepackage{multicol}

\ifx\pdfoutput\undefined
\usepackage{graphicx}
\else
\fi

\ifx\xetex
\usepackage{fontspec}
\usefont{EU1}{lmr}{m}{n}
\else
\fi

\usepackage{url} 
\usepackage{tikz}
\usetikzlibrary{matrix,arrows,positioning}
\tikzset{node distance=2cm, auto}

\usepackage[nodayofweek]{datetime}

\pdfoutput=1
\usepackage[activate={true,nocompatibility},final,kerning=true,tracking=true,spacing=true,factor=1100,stretch=10,shrink=10]{microtype}
\SetTracking{encoding={*}, shape=sc}{0}
\microtypecontext{spacing=nonfrench}

\parskip=\medskipamount

\newcommand{\conj}[1]{\quad\textnormal{ #1 }\quad}

\newcommand{\RHom}[0]{\ensuremath{\RR\!\Hom}}

\newcommand{\inp}[1]{\ensuremath{\langle #1 \rangle}}

\newcommand{\module}[0]{\operatorname{-mod}}

\newcommand{\normaltext}[1]{\textnormal{#1}}

\makeatletter
\def\imod#1{\allowbreak\mkern2.5mu({\operator@font mod}\,#1)}
\makeatother

\renewcommand{\a}{\alpha}
\renewcommand{\b}{\beta}

\newcommand{\opp}{\oplus}
\newcommand{\ott}{\otimes}

\newcommand{\ga}{\gamma}

\newcommand{\aA}{\mathcal{A}}

\newcommand{\aC}{\mathcal{C}}
\newcommand{\aD}{\mathcal{D}}
\newcommand{\aE}{\mathcal{E}}
\newcommand{\aF}{\mathcal{F}}

\newcommand{\aS}{\mathcal{S}}

\newcommand{\aW}{\mathcal{W}}

\newcommand{\rA}{\mathrm{A}}

\usepackage{bbm}

\newcommand{\LL}{\mathbb{L}}

\newcommand{\RR}{\mathbb{R}}

\newcommand{\ZZ}{\mathbb{Z}}

\newcommand{\eC}{\EuScript{C}}
\newcommand{\eD}{\EuScript{D}}
\newcommand{\eE}{\EuScript{E}}

\newcommand{\eY}{\EuScript{Y}}

\theoremstyle{plain}
\newtheorem{thm}[subsection]{Theorem}
\newtheorem{theorem}[subsection]{Theorem}

\newtheorem{prop}[subsection]{Proposition}
\theoremstyle{remark}

\newtheorem{rmk}[subsection]{Remark}
\newtheorem{remark}[subsection]{Remark}

\theoremstyle{definition}

\newtheorem{example}[subsection]{Example}

\newtheorem{defn}[subsection]{Definition}

\numberwithin{equation}{section}

\makeatletter
\def\imod#1{\allowbreak\mkern2.5mu({\operator@font mod}\,#1)}
\makeatother

\newcommand{\PA}{\Pi}

\newcommand{\Tri}{\Delta}

\newcommand{\osym}{\kappa}

\newcommand{\Ainf}{A_\infty}
\newcommand{\AAinf}{\aA_\infty}
\newcommand{\vnp}[1]{\lvert #1 \rvert}

\newcommand{\xto}[1]{\xrightarrow{#1}}

\usepackage{extarrows}

\newcommand{\dgcat}{\normaltext{dgcat}_k}
\newcommand{\Hqe}{\normaltext{Hqe}}

\theoremstyle{plain}

\newcommand{\Vect}{Vect}
\newcommand{\Ho}{H^0}
\newcommand{\CHo}{Ho}

\newcommand{\End}{End}

\newcommand{\Obj}{Ob}
\newcommand{\Ob}{\Obj}

\newcommand{\Hom}{Hom}
\newcommand{\Ext}{Ext}

\newcommand{\Ch}{Ch}

\newcommand{\proj}{P}
\newcommand{\tproj}{\proj} 
\newcommand{\simp}{S}
\newcommand{\tsimp}{Q}
\newcommand{\tM}{\tilde{M}}

\usepackage{soul} 

\begin{document}

\title[Preprojective Cones]{Preprojective analogue of the cone construction}
\author[Benjamin Cooper]{Benjamin Cooper}
\address{University of Iowa, Department of Mathematics, 14 MacLean Hall, Iowa City, IA 52242-1419}
\email{ben-cooper\char 64 uiowa.edu}
\author[Joshua Sussan]{Joshua Sussan}
\address{CUNY Medgar Evers, Department of Mathematics, 1650 Bedford Ave, Brooklyn, NY 11225}
\email{jsussan\char 64 mec.cuny.edu}

\begin{abstract}
  We formulate a relative, representation theoretic, notion of the algebraic
  cone construction. This motivates a generalization of the cone corresponding
  to a preprojective algebra.
\end{abstract}

\maketitle
\setcounter{tocdepth}{1}
\setcounter{secnumdepth}{2}

\section{Introduction}\label{introsec}
We give a a relative, representation theoretic, reformulation of the
algebraic cone construction. While variations have appeared in the
literature before \cite{KLH, Seidel, Soibelman} the version introduced here
is used to motivate a generalization of the cone construction obtained by
replacing the type $A$ quiver with its associated preprojective algebra.
This corresponds to the passage of a cone corresponding to a moduli space of
representations to one corresponding to a certain canonical bundle over the
moduli space, which is of importance in geometric representation theory
\cite[Thm. C]{Ringel}.

If $f : X \to Y$ is a chain map then the algebraic cone of $f$ is the chain complex 
$C(f) = X[1] \opp Y$ given by the direct sum of $X$ with $Y$ and differential
$$d_{C(f)} = \left(\begin{array}{cc} -d_X & 0 \\ f & d_Y \end{array} \right).$$
This familiar gadget is central to many important concepts in
mathematics.  In their study of differential graded algebras, Bondal and
Kapranov formulated the requirement that a map contain a cone in terms of
the completeness of the Yoneda embedding \cite{BK}. In creating the
foundations of $\Ainf$-categories, Kontsevich and Seidel found certain
$\Ainf$-categories $\Delta$ which corepresent distinguished triangles, in
particular these categories characterize algebraic cones up to homotopy \cite{KLH, Seidel}.
In this paper, we introduce a generalization of the criterion that a category
contains cones and distinguished triangles by combining these two ideas.

The remainder of the paper investigates this definition when the
$A_2$-quiver appearing in the Kontsevich-Seidel category $\Delta$ (see \S \ref{qrepsec}) is replaced by the preprojective algebra $\Pi Q$  of the $A_2$-quiver.
From a certain perspective, the preprojective algebra $\Pi Q$ is the simplest
non-trivial replacement for a quiver $Q$. In part, this is because $Q$
naturally embeds in $\Pi Q$, and as $Q$-module, $\Pi Q$ decomposes as a direct
sum of preprojective $Q$-modules. The moduli space of $\Pi Q$ finite
dimensional representations naturally fibers over the corresponding moduli
space of $Q$ representations and this relationship has important
consequences in representation theory.  For more details see \cite{Ringel}.

In order to understand the preprojective cone, in
Thm. \ref{bigcomputationtheorem} we use homotopy perturbation theory to compute
the minimal model $\Pi$ of the dg enhancement of  the $\Ext$-algebra
of indecomposable objects in the category $\Pi A_2\module$ where
$$\Pi = \Ext^*_{\Pi A_2}(I,I) \quad \quad  I = \bigoplus_{} M$$
and the direct sum is over all indecomposable $\Pi A_2$-modules (up to isomorphism).
This results in a complete description of the $\Ainf$-structure on a
tetrahedron with vertices labeled by indecomposable $\Pi A_2$-modules.
This $\Ainf$-structure allows us to study the derived mapping spaces
$\RHom(\Pi,\aD)$ as $\aD$ ranges over other $\Ainf$-categories. 
Since preprojective distinguished triangles are homotopy classes of such functors 
$[\tilde{F}]\in \Ob(\Ho(\RHom(\Pi,\aD)))$,
we are able to formulate conclusions about them from the minimal model for $\Pi$.

In short, a preprojective map is a pair of cycles
$$f : A \leftrightarrows B : g$$
and the preprojective distinguished triangle is constructed from the algebraic cones $C(f)$ and  $C(g)$ together with canonical maps between them
$$(12) : C(f) \leftrightarrows C(g) : (21).$$
In addition to the ordinary distinguished triangles which are determined
individually by $f$ and $g$, the preprojective distinguished triangle
contains families of Postnikov systems of arbitrary length. The
sense in which this data is precisely encoded by the $\Ainf$-structure
computed in Thm. \ref{bigcomputationtheorem} is discussed in Rmk. \ref{postnikovrmk}.

Finally, these observations are related to Fukaya categories of surfaces. In
Proposition \ref{aaekorelprop}, we conclude by constructing a strict
$\Ainf$-functor from the wrapped Fukaya category of a pair of paints
$\aW(P)$ to a slight modification of the preprojective category $\Pi$.

This paper is part of an ongoing study of the relativization of structures
in homological algebra and relationships to the categorification program in
low dimensional topology. It has been made into a separate paper because it
contains an explicit and lengthy computation of a certain $\Ainf$-structure.

\subsection*{Acknowledgments}
B.C. was supported by the University of Iowa Old Gold fellowship. J.S. was supported by NSF grant DMS-1407394, PSC-CUNY Award 67144-0045, and Simons Foundation Collaboration Grant 516673.

\section{The language of $A_\infty$-categories}\label{ainfsec}
An $A_\infty$-category is a category in which the associativity of
composition holds only up to coherent homotopy.  The purpose of this section
is to explain how the homotopy transfer theorem \cite{Markl} implies that
every differential graded category $\eC$ uniquely determines an
$A_\infty$-category $\Ho(\eC)$ and how homotopy classes of maps between dg
categories can be computed from this $A_\infty$-category \cite{Faonte}.

\subsection{$\Ainf$-categories and $\Ainf$-functors}\label{ainfstddefsec}

\begin{defn}\label{ainfdef}
An {\em $A_{\infty}$-category} $\aC$ consists of a collection of objects $\Ob(\aC)$ and a $\ZZ$-graded $k$-module of morphisms $\Hom(X,Y)=\bigoplus_{i \in \ZZ} \Hom^i(X,Y)$ for each pair of objects $X,Y\in \Ob(\aC)$ together with maps
$$m_d : \Hom(X_{d-1},X_{d})\ott  \cdots \ott\Hom(X_0,X_1)\to \Hom(X_0,X_d)[2-d],\quad d\geq 1$$
which satisfy the relations
\begin{equation}\label{ainfrel}
\sum_{l=0}^d \sum_{n=0}^{d-l} (-1)^{\ddagger_n} m_{d-l+1}(f_d,\ldots,f_{n+l+1},m_l(f_{n+l},\ldots,f_{n+1}),f_n,\ldots,f_1) = 0
\end{equation}
where $\ddagger_n = \vnp{f_n} + \cdots + \vnp{f_1}-n$ and $d \geq 1$.

An $A_\infty$-category $\aC$ is said to be {\em strictly unital} when there is a unique degree zero morphism $1_X \in \Hom^0(X,X)$ for each $X\in \Ob(\aC)$ which satisfies
\begin{align}\label{identityeqn}
m_2(f,1_X) = f, \quad\quad (-1)^{\vnp{g}}m_2(1_X,g) = g\\ \conj{ and }m_d(\ldots, 1_X, \ldots) = 0 \conj{ when } d \ne 2,\nonumber
\end{align}
for any maps $f : X \to A$ or $g : B \to X$ and any object $X\in\Ob(\aC)$.
\end{defn}  

\begin{example}\label{truncex}
Additive $k$-linear categories and differential graded categories are
examples of $\Ainf$-categories in which all of the higher
multiplications $m_d $, for $d > 2$, vanish.  Any $\Ainf$-category $\aC$ determines a dg category $\tau_{>2}\aC$ which is obtained by forgetting the maps $m_d$ for $d>2$.
Non-trivial examples of $\Ainf$-categories appear in Section \ref{preprojalgsec}. 
\end{example}

In an $A_\infty$-category $\aC$, there is a degree $1$ map, $m_1 : \Hom(X_0,X_1)\to \Hom(X_0,X_1)$, for each pair of objects $X_0,X_1\in \Ob(\aC)$, which satisfies $m_1\circ m_1 = 0$; the simplest $A_\infty$-relation above. Taking homology everywhere with respect to these maps produces the homotopy category defined below.

\begin{defn}\label{hodef}
The {\em homotopy category} $\Ho(\aC)$ of an $A_\infty$-category $\aC$ is the $k$-linear category with the same objects as $\aC$ and morphisms given by homology classes of maps $[f] \in H^*(\Hom(X,Y),m_1)$ for each $X,Y\in \Ob(\aC)$. The composition is defined by
$$[f_2]\circ [f_1] = (-1)^{\vnp{f_1}} [m_2(f_2,f_1)].$$
\end{defn}

\begin{example}\label{endex}
  Suppose that $R$ is ring and $M$ is an $R$-module. If $P$ is a projective resolution of $M$
$$ P = [\cdots \to P_i \xto{d_{i}} P_{i+1} \to \cdots \to P_{-1} \to P_0] \to M \to 0$$
then the endomorphisms $\End^*(P)$ of $P$ form a differential graded category with one object. 

In more detail, set $\End^n(P) = \prod_{i\in \ZZ} \Hom(P_i,P_{i+n})$. If $f = \{ f_i : P_i \to P_{i+n} \}$ and $g = \{ g_i : P_i \to P_{i+m} \}$ are maps of degree $n$ and $m$ respectively then the composite $g\circ f = \{ g_{n+i} \circ f_i : P_i \to P_{i+n+m}\}$ is an endomorphism of degree $n+m$.  This composition determines a multiplication $\mu$.  If $f = \{f_i : P_i \to P_{i+n} \}$ then $df = \{ (df)_i : P_i \to P_{i+n+1} \}$ where
$$(df)_{i} = d_{i+n} \circ f_{i} - (-1)^n f_{i+1} d_{i} : P_{i} \to P_{i+n+1}.$$

The homotopy category $\Ho(\End^*(P))$ of $\End^*(P)$ is the $\Ext$-algebra of $M$
$$\Ho(\End^*(P)) \cong \Ext^*_R(M).$$
\end{example}

\begin{defn}\label{functordef}
  An {\em $\Ainf$-functor} $F : \aC \to \aD$ between $A_\infty$-categories consists of a map $F : \Ob(\aC) \to \Ob(\aD)$ and multilinear maps 
$$F^d : \Hom_\aC(X_{d-1}, X_d) \ott \cdots \ott \Hom_\aC(X_0,X_1) \to \Hom_{\aD}(F(X_0), F(X_d))[1-d]$$
for $d\geq 1$, which satisfy the equations
\begin{align*}
\sum_{r\geq 1} \sum_{s_1,\ldots,s_r} m_r^{\aD}(F^{s_r}(f_d,\ldots,f_{d-s_r + 1}), \ldots, F^{s_1}(f_{s_1},\ldots, f_1)) = \\\sum_{l,n} (-1)^{\ddagger_n} F^{d-l+1}(f_d,\ldots, f_{n+l+1}, m^\aC_{l}(f_{n+l},\ldots, f_{n+1}),f_n,\ldots,f_1).
\end{align*}
The sign $\ddagger_n$ is as in Definition \ref{ainfdef} and the first sum is over all partitions: $s_1+\cdots +s_r = d$. The collection $\{F^d\}$ is also required to behave well with respect to units
$$F^1(1_X) = 1_{F(X)}\conj{ and } F^d(\ldots, 1_X, \ldots) = 0 \textnormal{ for } d\geq 2.$$

Any such $A_{\infty}$-functor $F : \aC \to \aD$ induces a map $\Ho(F) : \Ho(\aC) \to \Ho(\aD)$ between the associated homotopy categories. An $A_{\infty}$-functor is a {\em quasi-isomorphism} or {\em $\Ainf$-equivalence} when $\Ho(F)$ is an equivalence of categories.
\end{defn}

If $F : \aC \to \aD$ is an $\Ainf$-functor then there is a dg functor $\tau_{>2}(F) : \tau_{>2}\aC \to \tau_{>2}\aD$ determined by the map $F^1$ between the truncations of $\aC$ and $\aD$, see Example \ref{truncex}.

Two $\Ainf$-functors $F : \aC \to \aD$ and $G : \aD \to \aE$ can be composed to produce an $\Ainf$-functor $G \circ F : \aC \to \aE$, the $d$th component of which is given by the equation
$$(G\circ F)^d(f_d,\ldots,f_1) = \sum_{r} \sum_{s_1,\ldots,s_r} G^r(F^{s_r}(f_d,\ldots, f_{d-s_r+1}),\ldots, F^{s_1}(f_{s_1},\ldots, f_1)).$$

\begin{defn}\label{aainfdef}
Let $\AAinf(\aC,\aD)$ denote the $\Ainf$-category of $\Ainf$-functors from $\aC$ to $\aD$. If $F,G : \aC \to \aD $ are two objects in this category, then a morphism ({\em pre-natural transformation}) $T \in \Hom^g(F,G)$ of degree $g$ from $F$ to $G$ is a sequence $T = (T^0, T^1, \ldots)$ where
  $$T^d : \Hom_{\aC}(X_{d-1},X_d) \otimes \cdots \otimes \Hom_{\aC}(X_0,X_1) \to \Hom_{\aD}(F(X_{0}),G(X_d))[g-d],$$
  for all sequences of objects $(X_0,\ldots,X_d)$ in $\aC$.
There is an $\Ainf$-structure on the collection of pre-natural transformations. For more details see ~\cite[\S (1d)]{Seidel}.
\end{defn}

\subsection{Homotopy perturbation theory}\label{HPTsec}
Since an $\Ainf$-category is a category in which the
composition is associative only up to coherent homotopy, deforming an
$\Ainf$-category by a homotopy yields an $\Ainf$-equivalent
$\Ainf$-category. The purpose of the homotopy transfer theorem is to make
this precise, see \cite[Prop. 1.12]{Seidel} or \cite{Markl} for detailed
arguments and signs.


If $\aC$ and $\aD$ are dg categories with the same set of objects, then we
say that $\aD$ is a {\em perturbation} of $\aC$ when there are dg functors
$f : \aC \to \aD$ and $g : \aD \to \aC$ so that $f$ and $g$ are identity
maps on objects, $1_{\aD} = fg$, and for each pair of objects $x,y\in \aC$
there is a homotopy $h_{x,y} : \Hom_{\aC}(x,y) \rightarrow \Hom_{\aC}(x,y) $ of
degree $-1$ which satisfies $d h_{x,y} - h_{x,y} d = 1 - gf$ where $gf$ is
the map induced by $g$ and $f$ between $\Hom$-spaces.

\begin{thm}\label{HPthm}
  Suppose that $\aC$ is an $\Ainf$-category, $\tau_{>2}\aC$ is the dg
  category determined by forgetting the higher $m_n$-maps when $n \geq 3$ and
  $\aD$ is a perturbation of $\tau_{>2}\aC$ with dg functors
    $$f : \tau_{>2}\aC \rightarrow \aD\conj{ and } g : \aD \rightarrow \tau_{>2} \aC.$$
Then there is an $\Ainf$-category $(\aD',\{m_n^{\aD'}\}_{n>3})$ and there are $\Ainf$-functors 
$$f' : \aC \to \aD' \conj{ and } g' : \aD' \to \aC$$
which determine an $\Ainf$-equivalence  $\aC\simeq \aD'$ and restrict to the initial data:
$\tau_{>2}\aD' = \aD$, $\tau_{>2}(f') = f$ and $\tau_{>2}(g') = g$.
\end{thm}

Since $m_1^{\aD} = m_1^{\aD'}$ and $m_2^{\aD} = m_2^{\aD'}$, one can view $\aD'$ as an extension of the dg structure on $\aD$.


An important special case of this theorem shows that every $\Ainf$-category
$\aC$ is $\Ainf$-equivalent to its own homotopy category $\Ho(\aC)$,
\cite[Rmk 1.13]{Seidel}.  Following Example \ref{endex}, when $P$ is a
projective resolution of an $R$-module $M$, there is an $\Ainf$-structure on
the $\Ext$-algebra $\Ext^*(M)$ making it $\Ainf$-equivalent to the dg
algebra $\End^*(P)$ (with differential $d$ and multiplication $\mu$)
$$\End^*(P) \cong \Ext^*(M).$$

For this special case, following \cite[Theorem 3.2]{Keller}, the equivalence of Theorem \ref{HPthm} may be expressed using maps 
\begin{equation}\label{perturbdiag}
\begin{tikzpicture}[scale=10, node distance=2cm]
\node (A)  {$\End^*(P)$};
\node (B) [right of=A] {};
\node (C) [right of=B] {$\Ext^*(M)$};
\draw[transform canvas={yshift=+0.5ex},->] (A) to node {$p$} (C);
\draw[transform canvas={yshift=-0.5ex},->] (C) to node {$i$} (A);
\draw[->] (A) to [out=155,in=205,looseness=7] node [swap] {$H$} (A);
\end{tikzpicture}
\end{equation}
where $i$ and $p$ are morphisms of degree zero and $H$ is a homogeneous map of degree $-1$ such that
$$pi=1, \quad 1-ip=d(H), \quad H^2=0.$$
Then the $\Ainf$-structure on $\Ext^*(M)$ is given by
$$ m_n = \sum_T m_n^T $$
where $T$ ranges over planar rooted binary trees with $n$ leaves and
$m_n^T$ is given by composing the tree-shaped diagram obtained by labeling each leaf by $i$, each branch point by $\mu$, each internal edge by $H$ and the root by $p$.

For example, the two trees determining $m_3$ are pictured below.

\begin{equation}
\label{trees}
\begin{tikzpicture}[scale=10, node distance=1.5cm]
\node (A1)  {$i$};
\node (A2) [right of=A1] {};
\node (A3) [right of=A2] {$i$};
\node (A4) [right of=A3] {$i$};
\node (B2) [below of=A2] {$\mu$};
\node (B3) [below of=A3] {};
\node (B4) [below of=A4] {};
\node (C3) [below of=B3] {$\mu$};
\node (D3) [below of=C3] {$p$};
\draw[] (A1) to node {} (B2);
\draw[] (A3) to node {} (B2);
\draw[] (B2) to node {$H$} (C3);
\draw[] (A4) to node {} (C3);
\draw[] (C3) to node {} (D3);
\end{tikzpicture}
\begin{tikzpicture}[scale=10, node distance=1.5cm]
\node (A1)  {};
\node (A2) [right of=A1] {$i$};
\node (A3) [right of=A2] {$i$};
\node (A4) [right of=A3] {};
\node (A5) [right of=A4] {$i$};
\node (B4) [below of=A4] {$\mu$};
\node (B3) [below of=A3] {};
\node (B4) [below of=A4] {};
\node (C3) [below of=B3] {$\mu$};
\node (D3) [below of=C3] {$p$};
\draw[] (A3) to node {} (B4);
\draw[] (A2) to node {} (C3);
\draw[] (B4) to node {$H$} (C3);
\draw[] (A5) to node {} (B4);
\draw[] (C3) to node {} (D3);
\end{tikzpicture}
\end{equation}

\subsection{Homotopy classes of functors}
The category of differential graded categories $\dgcat$ over $k$ can be given the structure of a model category $\Hqe$. A weak equivalence in this homotopy theory is a dg functor $f : \eC \to \eD$ which satisfies two properties
\begin{enumerate}
\item $\Ho(f) : \Ho(\aC) \to \Ho(\aD)$ is essentially surjective
\item $f_{x,y} : \Hom_{\aC}(x,y) \to \Hom_{\aD}(f(x), f(y))$ is a quasi-isomorphism for all $x,y\in\Ob(\aC)$.
\end{enumerate}
The homotopy category $\CHo(\dgcat)$ of dg categories is the category obtained by formally inverting the weak equivalences above. To\"en \cite{Toen} proved that category $\CHo(\dgcat)$ is closed and monoidal. In particular, there are dg categories $\aC \ott^\LL \aD$ and $\RHom(\aD,\aE)$ together with natural isomorphisms
$$\Hom(\aC \ott^\LL \aD, \aE) \xto{\sim} \Hom(\aC, \RHom(\aD,\aE)).$$
Unfortunately, To\"{e}n's description of the dg category $\RHom(\eD,\eE)$ is somewhat complicated. The purpose of this section is to discuss an alternative way to compute the derived space of mappings using $\Ainf$-categories. G. Faonte proved the theorem below \cite[Thm. 1.7]{Faonte}; the statement is sometimes attributed to Kontsevich.

\begin{thm} \label{aainfthm}
If $\aD$ and $\aE$ are dg categories over a field of characteristic $0$, then the derived mapping category $\RHom(\aD,\aE)$ is naturally isomorphic to the dg category of $\Ainf$-functors from $\aD$ to $\aE$
$$\RHom(\aD,\aE) \xto{\sim} \AAinf(\aD,\aE).$$
\end{thm}

\section{Algebraic cones}\label{enrichedsec}
If $f : X \to Y$ is a continuous map between topological spaces $X$ and $Y$ then the cone $C(f)$ of $f$ is given by the pushout
$$C(f) = CX \sqcup_f Y \conj{ where } CX = X\times [0,1]/X\times \{1\}.$$
This topological cone acts as a stand in for the quotient $Y/im(f)$ in the long exact sequence of homology groups associated to the quotient
$$\cdots \to H_n(X) \xto{f_*} H_n(Y) \to H_n(C(f)) \to H_{n-1}(X) \to \cdots$$
because $H_n(C(f)) \cong H_n(Y/im(f))$ for $n> 0$.
After passing from topological spaces to cochain complexes,
the cone $C(f)$ of a map $f : (X,d_X) \to (Y,d_Y)$ is the cochain complex formed by
\begin{equation}\label{coneeq}
  C(f) = (X[1] \opp Y,d_{C(f)}) \conj{ where }  d_{C(f)} = \left(\begin{array}{cc} -d_X & 0  \\ f & d_Y \end{array} \right).
  \end{equation}
In analogy with the relationship between the homology of quotient space and the topological cone,
when $f : X \to Y$ is a map in an abelian category $\aA$,  the algebraic cone 
is isomorphic to the quotient in the derived category $D^b(\aA)$, $C(f) \cong Y/im(f)$. This is why the existence of objects equivalent to algebraic cones is a principal component of the definition of a triangulated category. For a similar discussion of cones see \cite[\S 3.1]{Auroux}.
In the remainder of this section we will reformulate the requirement that an algebraic cone exists in a manner which admits generalizations.

When an $\Ainf$-category $\aC$ is pretriangulated 
each cycle $f \in Hom_{\aC}(X,Y)$, has a cone $C(f) \in Ob(\aC)$. 
Roughly speaking, a cone is said to exist in $\aC$ when there is an object $C(f) \in \Ob(\aC)$ representing the cone of $f$ in the image of the Yoneda embedding.
Let us explain precisely what we mean. For any $\Ainf$-category $\aC$, the
category of modules $\aC\module=\AAinf(\aC,\Ch_k)$ consists of
$\Ainf$-functors from $\aC$ to the dg category of cochain complexes
$\Ch_k$. There is a Yoneda embedding of $\aC$ into its associated category
of modules $\aC\module$
$$\eY : \aC \to \aC\module, \conj{ } X \mapsto \eY_X\conj{ where } \eY_X(Y) = \Hom_{\aC}(X,Y).$$ 
This is an embedding in the sense that the associated functor $\Ho(\eY) : \Ho(\aC)\to\Ho(\aC\module)$ between homotopy categories is full and faithful.

\begin{defn}\label{conedef}
If $f \in\Hom_{\aC}(Y_0,Y_1)$ is a cycle, (so $m_1(f) = 0$), then the {\em cone $\eC(f)\in\Ob(\aC\module)$ of $f$} is the $\aC$-module determined by the assignment
$$\eC(f)(X) = \Hom(X,Y_0)[1] \opp \Hom(X,Y_1)$$
and structure maps
\begin{align*}
m^{\eC(f)}_d((b_0,b_1),a_{d-1},\ldots, a_1) = (m_d(b_0,a_{d-1},\ldots,a_1),\\ m_d(b_1,a_{d-1},\ldots,a_1) + m_{d+1}(f,b_0, a_{d-1},\ldots, a_1))
\end{align*}
\end{defn}

\begin{rmk}
The definition stems from the observation that for any $\Ainf$-category
$\aC$, the homotopy category of modules $\Ho(\aC\module)$ is triangulated in
the sense of Verdier and $\eC(f) \cong C(\eY(f))$ where
$\eY(f) : \eY_{Y_0} \to \eY_{Y_1}$.
\end{rmk}

\begin{rmk}
When $\aC = \Ch_k$, so $m_d=0$ for $d> 2$, 
the $\Ainf$-cone in Def. \ref{conedef} above is equivalent to the image of the algebraic cone in Eqn. \eqref{coneeq} under the Yoneda map. 
$$\eC(f) \cong \eY_{C(f)}.$$
\end{rmk}

This justifies the next definition.

\begin{defn}\label{inconedef}
  Suppose that $\aC$ is an $\Ainf$-category and $f \in\Hom_{\aC}(Y_0,Y_1)$ is a cycle, so that $m_1(f) = 0$,
then an object {\em $X\in\Ob(\aC)$ is a cone of $f$ in $\aC$} when there is an isomorphism
$$\eY_{X} \cong \eC(f)$$
in the homotopy category $\Ho(\aC\module)$ of $\aC$-modules. In particular,
{\em $f$ has a cone $C(f)$ in $\aC$} when there is such a cone object $X$ in
$\aC$. Since the Yoneda embedding is full and faithful up to homotopy, any
two cones of a single $f$ must be isomorphic in $\Ho(\aC)$.
\end{defn}

Use of the Yoneda embedding to characterize the cone construction in dg and $\Ainf$-categories appeared in \cite{BK, Soibelman}.

\subsection{A relative perspective on triangles}\label{reltrisec}
In what follows we review a different perspective of triangulated categories often attributed to Kontsevich, see \cite{KLH}, \cite[I, (3g)]{Seidel}.

If $f \in \Hom_{\aC}(X,Y)$ is a cycle of degree $\vnp{f} = 0$ and
$Z \in \Ob(\aC)$ is the cone on $f$, as above, then there are cycles
$g : Y \to Z$ and $h : Z \to X$ of degrees $\vnp{g} = 0$ and $\vnp{h} = 1$
respectively which is summarized by
$$ X\xto{f} Y \xto{g} Z \xto{h} X[1].$$
While any such collection of maps within an $\Ainf$-category $\aC$ could be called a triangle, a distinguished triangle must satisfy the additional properties below.
\begin{enumerate}
\item In the homotopy category, composing any two adjacent maps is zero
\begin{equation}\label{m2def}
  m_2(g, f) = 0, \quad m_2(h,g) = 0 \conj{ and } m_2(f,h) = 0.
\end{equation}
\item The Massey product of three consecutive maps is identity
\begin{equation}\label{m3def}
m_3(h,g,f) = 1_{X},\quad m_3(f,h,g) = 1_{Y} \conj{ and } m_3(g, f, h) = 1_Z.
\end{equation}
\end{enumerate}
It happens that these conditions suffice to distinguish distinguished
triangles. In particular, $Z\simeq C(f)$ when conditions (1) and (2) hold.
There is a category $\Tri$ which packages the information above in such a
way that $\Ainf$-functors $t : \Tri \to \aC$ from $\Tri$ to $\aC$ correspond
to distinguished triangles in $\aC$.

\begin{defn}\label{tricatdef}
  There is an $\Ainf$-category $\Tri$ which encodes  the constraints satisfied by a distinguished triangle. The objects are given by the set $\Ob(\Tri) =\{ A,B,C \}$ and maps
  are given by identity maps $1_A$,$1_B$ and $1_C$ together with maps
  $\a : A\to B$, $\b : B\to C$ and $\ga : C \to A$ of degrees
  $\vnp{\a} = 0$, $\vnp{\b}=0$ and $\vnp{\ga}=1$ respectively. 
$$\begin{tikzpicture}[scale=10, node distance=2cm]
\node (S_1) {$A$};
\node (X) [below of=S_1] {};
\node (A_2) [left=1cm of X] {$B$};
\node (S_2) [right=1cm of X] {$C$};
\draw[->] (S_1) to node[swap] {$\a$} (A_2);
\draw[->] (A_2) to node {$\b$} (S_2);
\draw[->] (S_2) to node[above] {$\ga$} (S_1);
\end{tikzpicture}$$
 The $\Ainf$-structure is determined by the requirements of strict unitality,
as in Def. \ref{ainfdef}, and Eqns \eqref{m2def} and \eqref{m3def} above. This is the partially wrapped Fukaya category of the disk with three marked points on the boundary, see Def. \ref{partfukdef} \cite{DK, HKK, Nadler}. 
\end{defn}

\begin{thm}\label{deltathm}
A triangle $X \xto{f} Y \xto{g} Z\xto{h} X[1]$ in $\Ho(\aC)$ is distinguished if and only if there is an $\Ainf$-functor $t : \Tri \to \aC$ such that 
\begin{enumerate}
\item $t(A) = X$, $t(B) = Y$ and $t(C) = Z$
\item $[t](\a) = f$, $[t](\b) = g$ and $[t](\ga) = h$ 
\end{enumerate}
where $[t] = \Ho(t) : \Ho(\Tri) \to \Ho(\aC)$.
\end{thm}

Informally, the theorem above says that $\Ainf$-functors correspond to
distinguished triangles in a given $\Ainf$-category $\aC$. 
This theorem will be used to rephrase the condition that an $\Ainf$-category  $\aC$ contains a cone object $C(f)$ for every cycle $f \in Hom_{\aC}(X,Y)$.


\subsection{$A_2$-representations as a subcategory}\label{qrepsec}

Since the third object $Z$ in a distinguished triangle is determined up to
isomorphism by its realization as the cone on the map $f : X\to Y$ between
the other two, the category $\Tri$ is Morita equivalent to the subcategory
consisting of two objects and a map between them
$$A_2 := [X \xto{f} Y].$$
This is called the $A_2$-quiver. A {\em module} or {\em representation} $M$ of $A_2$ consists of two $k$-vector spaces $V$ and $W$ assigned to the objects $X$ and $Y$ 
$$M(X) = V \conj{ and } M(Y) = W$$
together with a linear map $M(f) : V \to W$. In other words, a functor $A_2 \to \Vect_k$. A map $q : M \to N$ between two representations is a natural transformation. If $M$ and $N$ are two representations then there is a sum $M\opp N$ determined by the assignments $(M\opp N)(X) = M(X)\opp N(X)$, $(M\opp N)(Y) = M(Y)\opp N(Y)$ and  $(M\opp N)(f) = M(f) \opp N(f)$. If any isomorphism of the form $K \cong M \opp N$ implies that either $M=0$ or $N=0$ then the module $K$ is called {\em indecomposable}.

The representations of $A_2$ form the objects of an abelian category $A_2\module$ containing precisely three indecomposable objects: $P$, $S_1$ and $S_2$ determined by the table of functors below.
\begin{enumerate}
\item $P(X) = k$, $P(Y) = k$ and $P(f) = 1_k$.
\item $S_1(X) = k$, $S_1(Y) = 0$ and $S_1(f) = 0$.
\item $S_2(X) = 0$, $S_2(Y) = k$ and $S_2(f) = 0$.
\end{enumerate}

These modules form a short exact sequence
\begin{equation}\label{sesab}
  0 \to S_2 \to P \to S_1 \to 0
  \end{equation}
which is universal in the sense that the structure of the category $A_2\module$ is
  determined by \eqref{sesab} and the axioms of abelian categories.  The
  space $\Ext^1(S_1,S_2)$ is spanned by this extension. All of the other
  $\Ext$-groups vanish because $S_2$ and $P$ are projective. From Eqn.
  \eqref{sesab}, we see that the chain complex $Q = [S_2 \to P]$ is a projective resolution of $S_1$.
  Theorem \ref{HPthm} can be used compute the $\Ainf$-structure of the
  $\Ext$-algebra:
  $$\End^*(Q\opp P\opp S_2) \xto{\sim} \Ext^*(S_1\opp P \opp S_2)$$
  as in Example \ref{endex}. The $\Ainf$-structure on the righthand side of
  this equation is well-known, see \cite[App. B. 2]{Madsen}, up to sign conventions, it is identical to the
  category $\Tri$
$$\Tri \cong \Ext^*(S_1\opp P \opp S_2).$$

It is in this way that the $\Ainf$-category $\Tri$, the principal datum of a
triangulated category and the definition of algebraic cone stem directly
from the representation theory of the $A_2$-quiver.

\subsection{Cones as completions}\label{tricat}

The discussion above leads us to a generalization of
the requirement that an $\Ainf$-category have cones corresponding to each cycle.

\newcommand{\A}{A}
\renewcommand{\rA}{\bar{A}}

\begin{defn}\label{completedef}
Suppose that $\iota : \A\subset \rA$ is a pair of $\Ainf$-categories.
Then an $\Ainf$-category $\aC$ is $\iota$-complete when the pullback functor:
$\iota^* : \RHom(\rA,\aC) \to \RHom(\A,\aC)$ is quasi-essentially surjective (i.e. $\Ho(\iota^*)$ is essentially surjective).

If $\aC$ is $\iota$-complete then, up to homotopy, every functor $F : \A \to \aC$ from $\A$ into $\aC$ lifts to a functor $\tilde{F} : \rA \to \aC$ in such a way that the diagram below commutes.
$$\begin{tikzpicture}[scale=10, node distance=1.5cm]
\node (A1) {$\A$};
\node (C1) [below of=A1] {$\rA$};
\node (D1) [right of=A1] {$\aC$};
\draw[->] (A1) to node [swap] {$\iota$} (C1);
\draw[->] (A1) to node {$F$} (D1);
\draw[dashed,->] (C1) to node [swap] {$\tilde{F}$} (D1);
\end{tikzpicture}$$
\end{defn}

\begin{example}
Suppose that $\A$ is a diagram category consisting of a collection of disjoint points and $\rA= CA$ is the cone category obtained by adding an initial object. Then a category $\aC$ containing $\A$-limits will be $\iota$-complete (where $\iota : A \to CA$).
\end{example}  


In examples of most interest the functor $\iota^*$ is a
quasi-equivalence; both quasi-essentially surjective and quasi-fully faithful.

The definition below introduces some terminology to clarify the generalization we are making.

\begin{defn}
Suppose $\iota : A\to \bar{A}$ is a pair as in Def. \ref{completedef} above. A functor $F : A \to \eC$ is a {\em morphism} or {\em $\iota$-morphism}. A lift $\tilde{F} : \bar{A} \to \eC$ is a {\em $\iota$-distinguished triangle} associated $F$. For any morphism $F$, the additional information needed to define a lift $\tilde{F}$ is a {\em cone} or {\em $\iota$-cone} on $F$.
\end{defn}

\begin{example}\label{ibcompex}
Suppose $B$ is a finite dimensional algebra over $k$ and $\aS$ is the
collection of finite dimensional simple $B$-modules $\aS \subset B\module$. 
Set $S = \opp_{M\in\aS} M$. Then there are $\Ainf$-categories
$$\A_B = \End^*(S)\conj{ and } \rA_B = \End^*(S\opp B).$$ 
In the second term, $B$ is used to denote the algebra $B$ viewed as a left $B$-module over
itself.  
After applying Thm. \ref{HPthm} to both sides, the inclusion $\A_B \hookrightarrow \rA_B$ of dg algebras determines an $\Ainf$-functor $\iota : \A_B \to \rA_B$. In this sense, there is a notion of $\iota$-completeness associated to every finite dimensional algebra $B$.
\end{example}

When $B=A_2$, Thm. \ref{deltathm} combines with the discussion in Section \ref{qrepsec} to show that an $\Ainf$-category $\aC$ contains a cone $C(f)$ for every cycle $f\in\aC$ if and only if $\aC$ is $\iota$-complete as in the example above. 

Since there are many choices of $\iota$, it is important to limit investigation to interesting choices.
In the next section we will investigate this condition for the preprojective algebra of the $A_2$-quiver. For this category, it might make sense to call the $\iota$-distingished triangles, distinguished pyramids.

\newcommand{\palg}{\Pi Q}
\section{Preprojective cones}\label{preprojalgsec}

In this section we introduce the preprojective algebras $\palg$ and compute
the $\Ainf$-category $\PA$ associated to the derived endomorphisms of
indecomposable modules over $\Pi A_2$. The category $\PA$ constitutes our
generalization of the category $\Tri$ which was seen to describe
distinguished triangles in Theorem \ref{deltathm}.

A {\em quiver} is a finite directed graph $Q = (Q_0, Q_1)$ consisting of
{\em vertices} $Q_0$ and {\em edges} $Q_1$. Each edge $f\in Q_1$ has a {\em
  start} $s(f)\in Q_0$ and a {\em tail} $t(f)\in Q_0$. For example, $A_2$ in
Section \ref{qrepsec} is a quiver of the form $A_2 = (\{ X,Y\}, \{X \xto{f} Y\})$,
$t(f) = Y$ and $s(f) = X$.

Associated to any quiver $Q$, is a {\em path algebra} $kQ$ consisting of
$k$-linear combinations of paths between the vertices in $Q$. Any two such
paths $a, b\in kQ$ multiply by concatenation $ab$ when the vertex at which
$a$ ends agrees with the vertex at which $b$ begins; the product is defined
to be zero otherwise. The category of left modules over the path algebra
$kQ$ is equivalent to the category of functors $Q\module = \Hom(Q, \Vect_k)$
appearing in Section \ref{qrepsec}
$$kQ\module \cong Q\module.$$

The preprojective algebra $\palg$ is a different algebra associated to
the quiver $Q$. Given such a quiver $Q$, there is another quiver $\overline{Q}$
which is obtained by adding a formal inverse $f^* : t(f) \to s(f)$ to each
arrow $f : s(f) \to t(f)$. If the set of these arrows is denoted by $Q_1^*$ then $\overline{Q} = (Q_0, Q_1\cup Q_1^*)$. Let $\rho$ be the element of the path algebra $k\overline{Q}$ given by the sum
$$\rho = \sum_{f\in Q_1} (f f^* - f^* f).$$
The {\em preprojective algebra} is the quotient of the path algebra by the ideal generated by $\rho$
$$\palg = k\overline{Q}/(\rho).$$

\begin{rmk}
The category $\palg\module$ is a kind of
  next-simplest most-interesting replacement for the category $Q\module$, see \cite{CB} or \cite[Thm. C]{Ringel}.
\end{rmk}

\begin{example}
  When $Q = A_n$ is the graph consisting of $n$ vertices $\{ 1,\ldots, n\}$
  with one directed edge $(i,i+1) : i\to i+1$ for $1\leq i < n$. Then the
  graph $\overline{A}_n$ is formed by adding the inverses
  $(i,i+1)^* = (i+1,i) : i+1 \to i$ pictured below.
$$\begin{tikzpicture}[scale=10, node distance=2cm]
\node (S) {$1$};
\node (A) [right of=S] {$\cdots$};
\node (B) [right of=A] {$i-1$};
\node (C) [right of=B] {$\,\,i\,\,$};
\node (D) [right of=C] {$i+1$};
\node (E) [right of=D] {$\cdots$};
\node (F) [right of=E] {$n$};

\draw[->,bend left=25] (S) to node {} (A);
\draw[->,bend left=25] (A) to node {} (S);

\draw[->,bend left=25] (A) to node {} (B);
\draw[->,bend left=25] (B) to node {} (A);



\draw[->,bend left=25] (B) to node {} (C);
\draw[->,bend left=25] (C) to node {} (B);

\draw[->,bend left=25] (C) to node {} (D);
\draw[->,bend left=25] (D) to node {} (C);

\draw[->,bend left=25] (D) to node {} (E);
\draw[->,bend left=25] (E) to node {} (D);

\draw[->,bend left=25] (E) to node {} (F);
\draw[->,bend left=25] (F) to node {} (E);

\end{tikzpicture}$$




Quotienting the path algebra $k\overline{A}_n$ 
by the ideal $(\rho)$, described above, implies the relations of the preprojective algebra
\begin{align*}
(i,i-1)(i-1,i) &= (i,i+1)(i+1,i) \conj{ for } i = 2,\ldots,n-1,\\
& (1,2)(2,1) = 0 \conj{ and } (n,n-1)(n-1,n) = 0.
\end{align*}

These preprojective algebras are closely related to the algebras studied by Khovanov-Seidel: if $A_n^!$ denotes the Koszul dual of the Khovanov-Seidel algebra then the preprojective algebra is obtained by adding one relation
$$\Pi A_n = A_n^!/\inp{(n,n-1)(n-1,n)}$$
see \cite{KS} and \cite[\S 4]{QS}.
\end{example}

\subsection{The algebra $\Pi A_2$}\label{pitwosec}

In this section we will discuss the algebra $\Pi A_2$ and its representation theory in more detail.

\begin{defn}\label{pai2def}
  The algebra $\Pi A_2$ may be thought of as a category with two objects
  $\Ob(\Pi A_2) = \{1, 2\}$, each with its own identity map $1_1$ or $1_2$,
  and two maps $(12) : 1\to 2$ and $(21) : 2\to 1$ which satisfy two
  relations: 
$$(12)(21) = 0 \conj{ and } (21)(12) = 0.$$ 
The quiver underlying this construction is pictured below.
$$\begin{tikzpicture}[scale=10, node distance=3cm]
\node (S) {$1$};
\node (A) [right of=S] {$2$};
\draw[->,bend left=32] (S) to node {$(12)$} (A);
\draw[->,bend left=32] (A) to node {$(21)$} (S);
\end{tikzpicture}$$
\end{defn}

The representation theory of this algebra is well-known. The abelian category $\Pi A_2\module$ of finitely generated representations has four indecomposable modules: $\simp_1$, $\simp_2$, $\proj_1$ and $\proj_2$ \cite[\S 8]{Leclercetal}.  
The first two modules $\simp_1$ and $\simp_2$ are $1$-dimensional simple modules associated to the vertices $1$ and $2$, 
  $$\simp_1 = k(1) \conj{ and } \simp_2 = k(2)$$
 where $(i)$ acts as the identity on $S_i$ and all other basis elements of $\Pi A_2$ act trivially. 
The second two modules are projective modules spanned by the set of paths which begin at their respective vertices
$$\proj_1 = \Pi A_2 (1) \conj{ and } \proj_2 = \Pi A_2(2).$$
While there are no maps of degree zero between simple modules, the arrows $(12)$ and $(21)$ in the definition above induce maps $(12) : \proj_1 \to \proj_2$ and $(21) : \proj_2\to \proj_1$ between projective modules.  Each of the two maps, indicated by bold arrows in the diagram below, has a kernel and image, which give the two maps between projectives and simples also pictured in the diagrams below.
$$\begin{tikzpicture}[scale=10, node distance=1.5cm]
\node (S) {$\simp_1 \cong im (12)$};
\node (A) [right of=S] {};
\node (D) [right of=A] {$ker (12) \cong \simp_2$};
\node (B) [above of=A] {$\proj_2$};
\node (C) [below of=A] {$\proj_1$};
\draw[thick,->] (C) to node {} (B);
\draw[right hook->] (S) to node {$j_2$} (B);
\draw[->>] (C) to node {$p_1$} (S);
\draw[left hook->] (D) to node {$j_1$} (C);
\node (Z)  [right of=D] {};
\node (S2) [right of=Z] {$\simp_1 \cong ker (21)$};
\node (A2) [right of=S2] {};
\node (D2) [right of=A2] {$im (21) \cong \simp_2$};
\node (B2) [above of=A2] {$\proj_2$};
\node (C2) [below of=A2] {$\proj_1$};
\draw[thick,->] (B2) to node {} (C2);
\draw[right hook->] (S2) to node {$j_2$} (B2);
\draw[->>] (B2) to node {$p_2$} (D2);
\draw[left hook->] (D2) to node {$j_1$} (C2);
\end{tikzpicture}$$

There is an obvious symmetry implicit in the discussion above which we next record. 

\begin{prop}\label{osymprop}
There is an involution $\osym : \Pi A_2\module \to \Pi A_2\module$ induced by exchanging the indecomposable modules
$$\proj_1 \leftrightarrow \proj_2 \conj{ and } \simp_1 \leftrightarrow \simp_2$$
\end{prop}

The two short exact sequences
\begin{equation}\label{seseqn}
  0 \to \simp_1 \xto{j_2} \proj_2 \xto{p_2} \simp_2 \to 0\conj{ and } 0 \to \simp_2 \xto{j_1} \proj_1 \xto{p_1} \simp_1 \to 0
  \end{equation}
correspond to two maps $\a : \simp_2\to \simp_1$ and $\b : \simp_1 \to \simp_2$ which span the groups $\Ext^1(\simp_2,\simp_1)$ and $\Ext^1(\simp_1,\simp_2)$ respectively. 

There are projective resolutions $\tsimp_i$ of simple modules $\simp_i$ by the projective modules $\proj_1$ and $\proj_2$
$$\tsimp_i=[ \cdots \to \proj_1 \xto{(12)} \proj_2 \xto{(21)} \proj_1 \to \cdots \to \proj_i] \xto{p_i} \simp_i \to 0$$

The structure exhibited among the indecomposables of $\Pi A_2$ is a kind of
double of the structure discussed in Section \ref{qrepsec}. In this paper we
seek to understand what the $\Pi A_2$ analogue of $\Delta$ corepresents: a
preprojective analogue of distinguished triangles and algebraic cones. In order to answer this
question we construct an $\Ainf$-category analogue $\Pi$ of $\Delta$.

In more detail, we wish to understand the $\Ext$-algebra of
$M = \simp_1 \opp \proj_1 \opp \proj_2 \opp \simp_2$.  After replacing each simple $\simp_i$
with the projective resolution $\tsimp_i$, and setting
$\tM = \tsimp_1 \opp \tproj_1 \opp \tproj_2 \opp \tsimp_2$, Theorem \ref{HPthm}, allows us to
compute an $\Ainf$-structure on the $\Ext$-algebra $\Ext^*(M)$ so that the
dg category $\End^*(\tM)$ and $\Ext^*(M)$ are $\Ainf$-equivalent.
$$\End^*(\tM)\xto{\sim} \Ext^*(M)$$
The $\Ainf$-category $\Pi = \Ext^*(M)$ will serve as our replacement for
$\Delta$ in what follows.

\subsection{The category $\PA$}\label{picatsec}
As an $\Ainf$-category $\PA$ has $m_1 = 0$. The generating maps of the category $\PA$ are pictured below: 

$$\begin{tikzpicture}[scale=10, node distance=2cm]
\node (S_1) {$S_1$};
\node (A) [right of=S_1] {};
\node (B) [right of=A] {};
\node (C) [right of=B] {};
\node (S_2) [right of=C] {$\simp_2$};
\node (A_2) [above of=B] {$\proj_2$};
\node (A_1) [below of=B] {$\proj_1$};
\draw[->] (S_1) to node {$j_2$} (A_2);
\draw[->] (A_2) to node {$p_2$} (S_2);
\draw[->] (S_2) to node {$j_1$} (A_1);
\draw[->] (A_1) to node {$p_1$} (S_1);
\draw[transform canvas={xshift=-0.5ex},->] (A_2) to node[pos=.25,left] {$(21)$} (A_1);
\draw[transform canvas={xshift=0.5ex},->] (A_1) to node[pos=.25,right] {$(12)$} (A_2);
\draw[transform canvas={yshift=+0.5ex},->] (S_2) to node[pos=.75,above] {$\alpha$} (S_1);
\draw[transform canvas={yshift=-0.5ex},->] (S_1) to node[pos=.75,below] {$\beta$} (S_2);
\end{tikzpicture}$$
where $\vnp{\a} = 1$, $\vnp{\b} = 1$ and all other maps have degree $0$. 

\begin{remark}
  All of the tables in this section are written so that the left column is
  equivalent to $\osym$ of the right column; $\osym$ is defined in Proposition \ref{osymprop}.
\end{remark}

Let $u_1:= \a\b$ and $u_2 := \b\a$ denote the degree $2$ endomorphisms of
$\simp_1$ and $\simp_2$ respectively. Then the composition
$$m_2 : \Hom(X_1,X_2) \ott \Hom(X_0,X_1) \to \Hom(X_0,X_2),\conj{ } f\ott g \mapsto f\circ g$$
is determined by the requirements of the identity maps and the table below.
$$\begin{array}{llcll}
m_2(j_1,p_2) &= (21) &\quad  &  m_2(j_2,p_1) &= (12) \\
m_2(u_1^n,u_1^m) &= u_1^{n+m} &\quad  & m_2(u_2^n,u_2^m) &= u_2^{n+m}\\
m_2(\b u_1^m,u_1^n) &= \b u_1^{n+m} &\quad  & m_2(\a u_2^m,u_2^n) &= \a u_2^{n+m} \\
m_2(u_2^m, \b u_1^n) &= \b u_1^{n+m} &\quad & m_2(u_1^m, \a u_2^n) &= \a u_2^{n+m} \\
m_2(\a u_2^m, \b u_1^n) &= u_1^{n+m+1} &\quad  & m_2(\b u_1^m,\a u_2^n) &= u_2^{n+m+1}  \\
\end{array}$$

In other words, all of the compositions are zero besides those involving identity maps, $j_1p_2 = (21)$, $j_2p_1 = (12)$ and the maps $\a$ and $\b$ which generate a free subalgebra. Together with maps $u_1 = \a\b$ and $u_2 = \b\a$ there are relations
$$u_2\b = \b u_1 \conj{ and } u_1\a = \a u_2.$$


The generating set pictured above determines the basis for each $\Hom$-space  given below.
$$\begin{array}{lcl}
\Hom(\proj_1, \simp_1) = p_1 &\quad& \Hom(\proj_2, \simp_2) = p_2\\
\Hom(\simp_1,\proj_2) = j_2 &\quad& \Hom(\simp_2, \proj_1) = j_1\\
\Hom(\proj_1, \proj_2) = (12) &\quad& \Hom(\proj_2, \proj_1) = (21)\\
\Hom(\proj_1, \proj_1) = 1_{\proj_1} &\quad& \Hom(\proj_2, \proj_1) = 1_{\proj_2}\\
\Hom(\simp_1, \simp_1) = u_1^n &\quad& \Hom(\simp_2, \simp_2) = u_2^n\\
\Hom(\simp_1, \simp_2) = \b u_1^n &\quad& \Hom(\simp_2, \simp_1) = \a u_2^n\\
\end{array}$$
where $n\geq 0$.

When $d=3$, the $\Ainf$-multiplication map
$$m_3 : \Hom(X_2,X_3)\ott\Hom(X_1,X_2)\ott\Hom(X_0,X_1) \to \Hom(X_0,X_3)[-1]$$
is determined by the constraints of strict unitality (see \eqref{identityeqn}) and the behavior of certain triangles within the graph above under action of $m_2$. There are two basic triangles
\begin{equation}\label{basictri1eqn}
\begin{array}{clcl}
(A) & m_3(\a,p_2,j_2) = 1_{\simp_1} &\quad& m_3(\b,p_1,j_1) = 1_{\simp_2} \\
(B) & m_3(p_2,j_2,\a) = 1_{\simp_2} &\quad& m_3(p_1,j_1,\b) = 1_{\simp_1} \\
    & m_3(j_2,\a,p_2) = 1_{\proj_2} &\quad& m_3(j_1,\b,p_1) = 1_{\proj_1} 
\end{array}
\end{equation}
consisting of rotations of the upper and lower faces of the tetrahedron pictured above.

There is a trick to finding several other non-trivial $m_3$-products in $\PA$. They are implied by the $\Ainf$-relations and the basic triangles above. Ignoring signs for a moment, the first $\Ainf$-relation, $d=4$ in Eqn. \eqref{ainfrel}, to incorporate the $m_3$-operation is written in long form as follows:
\begin{align}\label{trickeqn}
  m_3(m_2(h,f_3),f_2,f_1) &+ m_3(h,m_2(f_3,f_2),f_1) + m_3(h,f_3,m_2(f_2,f_1))\\ & +  m_2(m_3(h,f_3,f_2),f_1) + m_2(h,m_3(f_3,f_2,f_1)) = 0.
\end{align}
So when all but the first and last terms in the sum vanish, each face $m_3(f_3,f_2,f_1) = g$ of the tetrahedron above gives rise to a number of other $m_3$-operations. These can be constructed by using 
non-trivial $m_2$-compositions on either the left
$$m_3(m_2(h,f_3),f_2,f_1) = m_2(h,m_3(f_3,f_2,f_1)) = m_2(h,g),$$
or, by symmetry, on the right
$$m_3(f_3,f_2,m_2(f_1,h)) = m_2(m_3(f_3,f_2,f_1),h) = m_2(g,h).$$

Using this trick, each case $(A)$ and $(B)$ above gives the three additional compositions below.
$$\begin{array}{cllcll}
(A) & m_3(\a,p_2,(12)) &= p_1 &\quad& m_3(\b,p_1,(21)) &= p_2 \\
    & m_3(u_2^n, p_2,j_2) &= \b u_1^{n-1} &\quad& m_3(u_1^n,p_1,j_1) &= \a u_2^{n-1} \\
    & m_3(\a u_2^n, p_2,j_2) &= u_1^n &\quad& m_3(\b u_1^n,p_1,j_1) &= u_2^n \\
(B) & m_3((21),j_2,\a) &= j_1 &\quad& m_3((12),j_1,\b) &= j_2 \\
    & m_3(p_2,j_2,u_1^n) &= \b u_1^{n-1} &\quad& m_3(p_1,j_1,u_2^n) &= \a u_2^{n-1} \\
    & m_3(p_2,j_2, \a u_2^n) &= u_2^n &\quad& m_3(p_1,j_1, \b u_1^{n}) &= u_1^n
\end{array}$$

When $d=4$, the $\Ainf$-multiplication map
$$m_4 : \Hom(X_3,X_4)\ott\Hom(X_2,X_3)\ott\Hom(X_1,X_2)\ott\Hom(X_0,X_1) \to \Hom(X_0,X_4)[-2]$$
is determined by the constraints of strict unitality (see Eqn. \eqref{identityeqn}) and the behavior of certain triangles within the graph above under action of $m_2$. The two basic operations below correspond to the left and right faces of the tetrahedron on the page pictured above.
\begin{equation}\label{basictri2eqn}
\begin{array}{clcl}
(A) & m_4(p_1,(21),j_2,u_1) = 1_{\simp_1}  &\quad& m_4(p_2,(12),j_1,u_2) = 1_{\simp_2} \\
(B) & m_4(u_1, p_1,(21),j_2) = 1_{\simp_1} &\quad& m_4(u_2,p_2,(12),j_1) = 1_{\simp_2} \\
    & m_4(j_2,u_1,p_1,(21)) = 1_{\proj_2}  &\quad& m_4(j_1,u_2,p_2,(12)) = 1_{\proj_1} \\
    & m_4((21),j_2,u_1,p_1) = 1_{\proj_1}  &\quad& m_4((12),j_1,u_2,p_2) = 1_{\proj_2} \\
\end{array}
\end{equation}
As explained for the $m_3$-operations above, due to the vanishing of some terms in the $\Ainf$-relation for the $m_4$-operation, we can act with the $m_2$-operation on either the left or the right of the basic triangles to obtain a few more $m_4$-operations. Each case, $(A)$ and $(B)$, determines three more $m_4$-operations, these are listed below.
$$\begin{array}{cllcll}
(A) & m_4(p_1,(21),j_2,u_1^{n+1}) &= u_1^n &\quad & m_4(p_2,(12),j_1,u_2^{n+1}) &= u_2^n\\
    & m_4(p_1,(21),j_2,\a u_2^{n+1}) &= \a u_2^n &\quad & m_4(p_2,(12),j_1, \b u_1^{n+1}) &= \b u_1^n\\
    & m_4((12),(21), j_2, u_1) &= j_2 &\quad& m_4((21),(12),j_1,u_2) &= j_1 \\
(B) & m_4(u_1^{n+1},p_1,(21),j_2) &= u_1^n &\quad& m_4(u_2^{n+1},p_2,(12),j_1) &= u_2^n \\
    & m_4( u_2^{n+1} \b,p_1,(21),j_2) &=  u_2^n \b  &\quad& m_4( u_1^{n+1} \a ,p_2,(12),j_1) &=  u_1^n \a \\
    & m_4(u_1,p_1,(21),(12)) &= p_1 &\quad& m_4(u_2,p_2,(12),(21)) &= p_2 \\
\end{array}$$

In order to facilitate our description of the rest of the $\Ainf$-structure, we will use the following notation
$$(212) = (21) \ott (12)\conj{and} (121) = (12) \ott (21).$$
For instance,
$$(212)^n = (21) \ott (12) \ott (21) \ott (12) \ott \cdots \ott (21) \ott (12) \conj{ $n$-times. }$$

The classification of higher homotopies (see \eqref{higherhs} later on) shows that the only non-zero higher operations must contain an expression of the form:
$$\begin{array}{lcl}
p_2 \ott (121)^n\ott j_2  & \quad & p_1 \ott (212)^n\ott j_1 \\
p_2 \ott (121)^{n+1} & \quad & p_1 \ott (212)^{n+1} \\
p_2 \ott (121)^n \ott (12) & \quad & p_1 \ott (212)^n \ott (21)\\
p_2 \ott (121)^n \ott (12) \ott j_1 & \quad & p_1 \ott (212)^n \ott (21) \ott j_2 \\
\end{array}$$

\begin{rmk}\label{rmkev}
Intuitively speaking, if we view the input of an $\Ainf$-operation $m_n(f_n,\ldots,f_1)$ as a path $f_1,\ldots,f_n$ in the graph featured at the beginning of \S\ref{picatsec} then the higher $\Ainf$-operations that we find can been seen as extensions of the lower order operations discussed above. Each of these extensions is formed by adding a loop of the form $(121)$ or $(212)$ to the path while balancing the grading by adding a loop of the form $u_1=\a\b$ or $u_2 = \b\a$. The equations above list the ways in which the loops $(121)$ or $(212)$ can be added. See Rmk. \ref{rmkev2}.
\end{rmk}

This classification result is accomplished by Thm. \ref{bigcomputationtheorem}, which is a computation using homotopy perturbation theory (Thm. \ref{HPthm}). We need to introduce a few more preliminaries before proceeding to the theorem.

Our goal in Thm. \ref{bigcomputationtheorem} is to compute all of the compositions of maps corresponding to the decorated binary trees discussed in \S \ref{HPTsec}.
The key is to construct homotopies "$h$-maps" for compositions which are nullhomotopic, and study compositions of these homotopies.
We now define certain important maps in $\End^*(\tM)$.  In what follows below, $i \in \mathbb{Z}/2$.

Below is the map $h_{\simp_i \simp_i}^{(n)} \colon \simp_i \rightarrow \simp_i [-2n]$.
$$\begin{tikzpicture}[scale=10, node distance=1.5cm]
\node(A0) {$\proj_i$};
\node(A1)[left of=A0] {$\proj_{i+1}$};
\node(A2)[left of=A1] {$\proj_i$};
\node(A3)[left of=A2] {$\cdots$};
\draw[->] (A1) to node {} (A0);
\draw[->] (A2) to node {} (A1);
\draw[->] (A3) to node {} (A2);
\node(B0)[below of=A0] {$\proj_i$};
\node(B1)[below of=A1] {$\proj_{i+1}$};
\node(B2)[below of=A2] {$\proj_i$};
\node(B3)[below of=A3] {$\cdots$};

\draw[->] (B1) to node {} (B0);
\draw[->] (B2) to node {} (B1);
\draw[->] (B3) to node {} (B2);

\draw[->] (A0) to node {$1$} (B0);
\draw[->] (A1) to node {$1$} (B1);
\draw[->] (A2) to node {$1$} (B2);

\node(C1)[right of=B0] {$\proj_{i+1}$};
\node(C2)[right of=C1] {$\cdots$};
\node(C3)[right of=C2] {$\proj_{i+1}$};
\node(C4)[right of=C3] {$\proj_i$};

\draw[->] (B0) to node {} (C1);
\draw[->] (C1) to node {} (C2);
\draw[->] (C2) to node {} (C3);
\draw[->] (C3) to node {} (C4);

\end{tikzpicture}$$

Below is the map $h_{\simp_i \simp_{i+1}}^{(n)} \colon \simp_{i+1} \rightarrow \simp_i [-2n-1]$.
$$\begin{tikzpicture}[scale=10, node distance=1.5cm]
\node(A0) {$\proj_{i+1}$};
\node(A1)[left of=A0] {$\proj_i$};
\node(A2)[left of=A1] {$\proj_{i+1}$};
\node(A3)[left of=A2] {$\cdots$};
\draw[->] (A1) to node {} (A0);
\draw[->] (A2) to node {} (A1);
\draw[->] (A3) to node {} (A2);
\node(B0)[below of=A0] {$\proj_{i+1}$};
\node(B1)[below of=A1] {$\proj_i$};
\node(B2)[below of=A2] {$\proj_{i+1}$};
\node(B3)[below of=A3] {$\cdots$};

\draw[->] (B1) to node {} (B0);
\draw[->] (B2) to node {} (B1);
\draw[->] (B3) to node {} (B2);

\draw[->] (A0) to node {$1$} (B0);
\draw[->] (A1) to node {$1$} (B1);
\draw[->] (A2) to node {$1$} (B2);

\node(C1)[right of=B0] {$\proj_i$};
\node(C2)[right of=C1] {$\cdots$};
\node(C3)[right of=C2] {$\proj_{i+1}$};
\node(C4)[right of=C3] {$\proj_i$};

\draw[->] (B0) to node {} (C1);
\draw[->] (C1) to node {} (C2);
\draw[->] (C2) to node {} (C3);
\draw[->] (C3) to node {} (C4);

\end{tikzpicture}$$

Below is the map $h_{\simp_i \proj_i}^{(n)} \colon \proj_i \rightarrow \simp_i [-2n]$.
$$\begin{tikzpicture}[scale=10, node distance=1.5cm]
\node(A0){$\proj_i$};
\node(A1)[left of=A0]{$\cdots$};
\node(A2)[left of=A1]{$\proj_i$};
\node(A3)[left of=A2]{$\proj_{i+1}$};
\node(A4)[left of=A3]{$\proj_i$};
\node(A5)[left of=A4]{$\proj_{i+1}$};
\node(A6)[left of=A5]{$\cdots$};
\node(B4)[above of=A4]{$\proj_i$};

\draw[->] (A1) to node {} (A0);
\draw[->] (A2) to node {} (A1);
\draw[->] (A3) to node {} (A2);
\draw[->] (A4) to node {} (A3);
\draw[->] (A5) to node {} (A4);
\draw[->] (A6) to node {} (A5);
\draw[->] (B4) to node {$1$} (A4);

\end{tikzpicture}$$

Below is the map $h_{\simp_i \proj_{i+1}}^{(n)} \colon \proj_{i+1} \rightarrow \simp_i [-2n-1]$.
$$\begin{tikzpicture}[scale=10, node distance=1.5cm]
\node(A0){$\proj_i$};
\node(A1)[left of=A0]{$\cdots$};
\node(A2)[left of=A1]{$\proj_{i+1}$};
\node(A3)[left of=A2]{$\proj_i$};
\node(A4)[left of=A3]{$\proj_{i+1}$};
\node(A5)[left of=A4]{$\proj_i$};
\node(A6)[left of=A5]{$\cdots$};
\node(B4)[above of=A4]{$\proj_{i+1}$};

\draw[->] (A1) to node {} (A0);
\draw[->] (A2) to node {} (A1);
\draw[->] (A3) to node {} (A2);
\draw[->] (A4) to node {} (A3);
\draw[->] (A5) to node {} (A4);
\draw[->] (A6) to node {} (A5);
\draw[->] (B4) to node {$1$} (A4);

\end{tikzpicture}$$

Below is the map $h_{\proj_i \simp_i}^{(n)} \colon \simp_i \rightarrow \proj_i[2n] $.
$$\begin{tikzpicture}[scale=10, node distance=1.5cm]
\node(C0){$\cdots$};
\node(C1)[right of=C0]{$\proj_i$};
\node(C2)[right of=C1]{$\proj_{i+1}$};
\node(C3)[right of=C2]{$\proj_i$};
\node(C4)[right of=C3]{$\proj_{i+1}$};
\node(C5)[right of=C4]{$\cdots$};
\node(C6)[right of=C5]{$\proj_{i}$};
\node(B3)[below of=C3]{$\proj_i$};

\draw[->] (C0) to node {} (C1);
\draw[->] (C1) to node {} (C2);
\draw[->] (C2) to node {} (C3);
\draw[->] (C3) to node {$1$} (B3);
\draw[->] (C3) to node {} (C4);
\draw[->] (C4) to node {} (C5);
\draw[->] (C5) to node {} (C6);

\end{tikzpicture}$$

Below is the map $h_{\proj_{i+1} \simp_i}^{(n)} \colon \simp_i \rightarrow \proj_{i+1}[2n+1] $.
$$\begin{tikzpicture}[scale=10, node distance=1.5cm]
\node(C0){$\cdots$};
\node(C1)[right of=C0]{$\proj_i$};
\node(C2)[right of=C1]{$\proj_{i+1}$};
\node(C3)[right of=C2]{$\proj_i$};
\node(C4)[right of=C3]{$\proj_{i+1}$};
\node(C5)[right of=C4]{$\cdots$};
\node(C6)[right of=C5]{$\proj_{i}$};
\node(B4)[below of=C4]{$\proj_{i+1}$};

\draw[->] (C0) to node {} (C1);
\draw[->] (C1) to node {} (C2);
\draw[->] (C2) to node {} (C3);
\draw[->] (C4) to node {$1$} (B4);
\draw[->] (C3) to node {} (C4);
\draw[->] (C4) to node {} (C5);
\draw[->] (C5) to node {} (C6);

\end{tikzpicture}$$

The trees appearing in \eqref{trees} propagate from leaves to root.  The homotopies appear in order determined by distance from the leaves.  The initial homotopies $H$ arise as follows. 

\begin{align}
\label{initialhs}
H(p_1 j_1) &= h_{\simp_1 \simp_2}^{(0)} &\quad  H(p_2 j_2) &= h_{\simp_2 \simp_1}^{(0)}\\ \nonumber
H(j_1 (\beta u_1^n) ) &= h_{\proj_1 \simp_1}^{(n)} &\quad H(j_2 (\alpha u_2^n)) &= h_{\proj_2 \simp_2}^{(n)} \\ \nonumber
H(p_1(21)) &= h_{\simp_1 \proj_2}^{(0)} &\quad H(p_2(12)) &= h_{\simp_2 \proj_1}^{(0)} \nonumber \\
H(j_2 u_1^n) &= h_{\proj_2 \simp_1}^{(n-1)},  \quad n>0 &\quad  H(j_1 u_2^n) &= h_{\proj_1 \simp_2}^{(n-1)}, \quad n>0 \nonumber
\end{align}

"Higher" $h$-maps arise by applying $H$ to these initial $h$ maps as follows.
\begin{align}
\label{higherhs}
H(h_{\simp_1 \proj_2}^{(n)} \circ j_2) &= h_{\simp_1 \simp_1}^{(n+1)}, \quad n \geq 0 &\quad
H(h_{\simp_2 \proj_1}^{(n)} \circ j_1) &= h_{\simp_2 \simp_2}^{(n+1)}, \quad n \geq 0\\ \nonumber
H(h_{\simp_1 \proj_2}^{(n)} \circ (12)) &= h_{\simp_1 \proj_1}^{(n+1)}, \quad n \geq 0 &\quad
H(h_{\simp_2 \proj_1}^{(n)} \circ (21)) &= h_{\simp_2 \proj_2}^{(n+1)}, \quad n \geq 0\\ \nonumber
H(h_{\simp_1 \proj_1}^{(n)} \circ j_1) &= h_{\simp_1 \simp_2}^{(n)}, \quad n \geq 1 &\quad
H(h_{\simp_2 \proj_2}^{(n)} \circ j_2) &= h_{\simp_2 \simp_1}^{(n)}, \quad n \geq 1\\ \nonumber
H(h_{\simp_1 \proj_1}^{(n)} \circ (21)) &= h_{\simp_1 \proj_2}^{(n)}, \quad n \geq 1 &\quad
H(h_{\simp_2 \proj_2}^{(n)} \circ (12)) &= h_{\simp_2 \proj_1}^{(n)}, \quad n \geq 1 \\ \nonumber
H((12) \circ h_{\proj_1 \simp_1}^{(n)}) &= h_{\proj_2 \simp_1}^{(n-1)}, \quad n \geq 1 &\quad
H((21) \circ h_{\proj_2 \simp_2}^{(n)}) &= h_{\proj_1 \simp_2}^{(n-1)}, \quad n \geq 1 \\ \nonumber
H((21) \circ h_{\proj_2 \simp_1}^{(n)}) &= h_{\proj_1 \simp_1}^{(n)}, \quad n \geq 0 &\quad
H((12) \circ h_{\proj_1 \simp_2}^{(n)}) &= h_{\proj_2 \simp_2}^{(n)}, \quad n \geq 0.
\end{align}
Compositions of $h$'s with elements in the $Ext$-algebra and other $h$'s produce the elements in the $Ext$-algebra listed below. The map $H$ is zero on any element $f$ representing a cycle in the $Ext$-algebra.
\begin{align}
\label{hstoext}
u_1^n \circ h_{\simp_1 \simp_2}^{(m)} &= u_1^{n-m-1} \alpha &\quad
u_2^n \circ h_{\simp_2 \simp_1}^{(m)} &= u_2^{n-m-1} \beta \\ \nonumber
u_2^n \beta \circ h_{\simp_1 \simp_2}^{(m)} &= u_2^{n-m}  &\quad
u_1^n \alpha \circ h_{\simp_2 \simp_1}^{(m)} &= u_1^{n-m}  \\ \nonumber
h_{\proj_1 \simp_1}^{(0)} \circ p_1 &= 1_{\proj_1} &\quad
h_{\proj_2 \simp_2}^{(0)} \circ p_2 &= 1_{\proj_2} \\ \nonumber
(12) \circ h_{\proj_1 \simp_1}^{(0)} &= j_2 &\quad
(21) \circ h_{\proj_2 \simp_2}^{(0)} &= j_1 \\ \nonumber
u_2^m \beta \circ h_{\simp_1 \proj_2}^{(n)} &=\delta_{m,n} p_2 &\quad
u_1^m \alpha \circ h_{\simp_2 \proj_1}^{(n)} &=\delta_{m,n} p_1 \\ \nonumber
u_1^n \circ h_{\simp_1 \simp_1}^{(m)} &= u_1^{n-m} &\quad
u_2^n \circ h_{\simp_2 \simp_2}^{(m)} &= u_2^{n-m} \\ \nonumber
u_1^n \circ h_{\simp_1 \proj_1}^{(m)} &= \delta_{n,m} p_1 &\quad
u_2^n \circ h_{\simp_2 \proj_2}^{(m)} &= \delta_{n,m} p_2 \\ \nonumber
u_2^k \beta \circ h_{\simp_1 \simp_1}^{(n)} &= u_2^{k-n} \beta &\quad
u_1^k \alpha \circ h_{\simp_2 \simp_2}^{(n)} &= u_1^{k-n} \alpha \\ \nonumber
h_{\proj_2 \simp_1}^{(m)} \circ h_{\simp_1 \proj_2}^{(n)} &=\delta_{m,n} 1_{\proj_2} &\quad
h_{\proj_1 \simp_2}^{(m)} \circ h_{\simp_2 \proj_1}^{(n)} &=\delta_{m,n} 1_{\proj_1} \\ \nonumber
h_{\proj_1 \simp_1}^{(m)} \circ h_{\simp_1 \proj_1}^{(n)} &=\delta_{m,n} 1_{\proj_1} &\quad
h_{\proj_2 \simp_2}^{(m)} \circ h_{\simp_2 \proj_2}^{(n)} &=\delta_{m,n} 1_{\proj_2}.
\end{align}
We also have the following list of "partial" terminating operations.
Let $\gamma_1=\a$ and $\gamma_2=\b$ and let $i \in \mathbb{Z}/2$. These compositions are not cycles, so the map $H$ is defined to be zero on them.

Below is the map $h_{\simp_i \simp_i}^{(n)} \circ u_i^k$.
$$\begin{tikzpicture}[scale=10, node distance=1.5cm]
\node(A0) {$\proj_i$};
\node(A1)[left of=A0] {$\proj_{i+1}$};
\node(A2)[left of=A1] {$\proj_i$};
\node(A3)[left of=A2] {$\cdots$};
\node(A4)[right of=A0] {$\cdots$};
\node(A5)[right of=A4] {$\proj_i$};
\draw[->] (A1) to node {} (A0);
\draw[->] (A2) to node {} (A1);
\draw[->] (A3) to node {} (A2);
\draw[->] (A0) to node {} (A4);
\draw[->] (A4) to node {} (A5);

\node(B0)[below of=A0] {$\proj_i$};
\node(B1)[below of=A1] {$\proj_{i+1}$};
\node(B2)[below of=A2] {$\proj_i$};
\node(B3)[below of=A3] {$\cdots$};

\draw[->] (B1) to node {} (B0);
\draw[->] (B2) to node {} (B1);
\draw[->] (B3) to node {} (B2);

\draw[->] (A0) to node {$1$} (B0);
\draw[->] (A1) to node {$1$} (B1);
\draw[->] (A2) to node {$1$} (B2);

\node(C1)[right of=B0] {$\proj_{i+1}$};
\node(C2)[right of=C1] {$\cdots$};
\node(C3)[right of=C2] {$\proj_{i+1}$};
\node(C4)[right of=C3] {$\proj_i$};

\draw[->] (B0) to node {} (C1);
\draw[->] (C1) to node {} (C2);
\draw[->] (C2) to node {} (C3);
\draw[->] (C3) to node {} (C4);
\end{tikzpicture}$$

Below is the map $h_{\simp_i \simp_i}^{(n)} \circ u_i^k \gamma_i$.
$$\begin{tikzpicture}[scale=10, node distance=1.5cm]
\node(A0) {$\proj_i$};
\node(A1)[left of=A0] {$\proj_{i+1}$};
\node(A2)[left of=A1] {$\proj_i$};
\node(A3)[left of=A2] {$\cdots$};
\node(A4)[right of=A0] {$\cdots$};
\node(A5)[right of=A4] {$\proj_{i+1}$};
\draw[->] (A1) to node {} (A0);
\draw[->] (A2) to node {} (A1);
\draw[->] (A3) to node {} (A2);
\draw[->] (A0) to node {} (A4);
\draw[->] (A4) to node {} (A5);

\node(B0)[below of=A0] {$\proj_i$};
\node(B1)[below of=A1] {$\proj_{i+1}$};
\node(B2)[below of=A2] {$\proj_i$};
\node(B3)[below of=A3] {$\cdots$};

\draw[->] (B1) to node {} (B0);
\draw[->] (B2) to node {} (B1);
\draw[->] (B3) to node {} (B2);

\draw[->] (A0) to node {$1$} (B0);
\draw[->] (A1) to node {$1$} (B1);
\draw[->] (A2) to node {$1$} (B2);

\node(C1)[right of=B0] {$\proj_{i+1}$};
\node(C2)[right of=C1] {$\cdots$};
\node(C3)[right of=C2] {$\proj_{i+1}$};
\node(C4)[right of=C3] {$\proj_i$};

\draw[->] (B0) to node {} (C1);
\draw[->] (C1) to node {} (C2);
\draw[->] (C2) to node {} (C3);
\draw[->] (C3) to node {} (C4);

\end{tikzpicture}$$

Below is the map $h_{\simp_i \simp_{i+1}}^{(n)} \circ u_{i+1}^k $.
$$\begin{tikzpicture}[scale=10, node distance=1.5cm]
\node(A0) {$\proj_{i+1}$};
\node(A1)[left of=A0] {$\proj_{i}$};
\node(A2)[left of=A1] {$\proj_{i+1}$};
\node(A3)[left of=A2] {$\cdots$};
\node(A4)[right of=A0] {$\cdots$};
\node(A5)[right of=A4] {$\proj_{i+1}$};
\draw[->] (A1) to node {} (A0);
\draw[->] (A2) to node {} (A1);
\draw[->] (A3) to node {} (A2);
\draw[->] (A0) to node {} (A4);
\draw[->] (A4) to node {} (A5);

\node(B0)[below of=A0] {$\proj_{i+1}$};
\node(B1)[below of=A1] {$\proj_{i}$};
\node(B2)[below of=A2] {$\proj_{i+1}$};
\node(B3)[below of=A3] {$\cdots$};

\draw[->] (B1) to node {} (B0);
\draw[->] (B2) to node {} (B1);
\draw[->] (B3) to node {} (B2);

\draw[->] (A0) to node {$1$} (B0);
\draw[->] (A1) to node {$1$} (B1);
\draw[->] (A2) to node {$1$} (B2);

\node(C1)[right of=B0] {$\proj_{i}$};
\node(C2)[right of=C1] {$\cdots$};
\node(C3)[right of=C2] {$\proj_{i+1}$};
\node(C4)[right of=C3] {$\proj_i$};

\draw[->] (B0) to node {} (C1);
\draw[->] (C1) to node {} (C2);
\draw[->] (C2) to node {} (C3);
\draw[->] (C3) to node {} (C4);
\end{tikzpicture}$$

Below is the map $h_{\simp_i \simp_{i+1}}^{(n)} \circ u_{i+1}^k \gamma_{i+1}$.
$$\begin{tikzpicture}[scale=10, node distance=1.5cm]
\node(A0) {$\proj_{i+1}$};
\node(A1)[left of=A0] {$\proj_{i}$};
\node(A2)[left of=A1] {$\proj_{i+1}$};
\node(A3)[left of=A2] {$\cdots$};
\node(A4)[right of=A0] {$\cdots$};
\node(A5)[right of=A4] {$\proj_{i}$};
\draw[->] (A1) to node {} (A0);
\draw[->] (A2) to node {} (A1);
\draw[->] (A3) to node {} (A2);
\draw[->] (A0) to node {} (A4);
\draw[->] (A4) to node {} (A5);

\node(B0)[below of=A0] {$\proj_{i+1}$};
\node(B1)[below of=A1] {$\proj_{i}$};
\node(B2)[below of=A2] {$\proj_{i+1}$};
\node(B3)[below of=A3] {$\cdots$};

\draw[->] (B1) to node {} (B0);
\draw[->] (B2) to node {} (B1);
\draw[->] (B3) to node {} (B2);

\draw[->] (A0) to node {$1$} (B0);
\draw[->] (A1) to node {$1$} (B1);
\draw[->] (A2) to node {$1$} (B2);

\node(C1)[right of=B0] {$\proj_{i}$};
\node(C2)[right of=C1] {$\cdots$};
\node(C3)[right of=C2] {$\proj_{i}$};
\node(C4)[right of=C3] {$\proj_{i}$};

\draw[->] (B0) to node {} (C1);
\draw[->] (C1) to node {} (C2);
\draw[->] (C2) to node {} (C3);
\draw[->] (C3) to node {} (C4);

\end{tikzpicture}$$

Below is the map $h_{\simp_i \proj_{i+1}}^{(n)} \circ h_{\proj_{i+1} \simp_i}^{(k)}$.
$$\begin{tikzpicture}[scale=10, node distance=1.5cm]
\node(A0) {$\proj_{i+1}$};
\node(A1)[left of=A0] {$\proj_{i}$};
\node(A2)[left of=A1] {$\proj_{i+1}$};
\node(A3)[left of=A2] {$\cdots$};
\node(A4)[right of=A0] {$\cdots$};
\node(A5)[right of=A4] {$\proj_{i}$};
\draw[->] (A1) to node {} (A0);
\draw[->] (A2) to node {} (A1);
\draw[->] (A3) to node {} (A2);
\draw[->] (A0) to node {} (A4);
\draw[->] (A4) to node {} (A5);

\node(B0)[below of=A0] {$\proj_{i+1}$};
\node(B1)[below of=A1] {$\proj_{i}$};
\node(B2)[below of=A2] {$\proj_{i+1}$};
\node(B3)[below of=A3] {$\cdots$};

\draw[->] (B1) to node {} (B0);
\draw[->] (B2) to node {} (B1);
\draw[->] (B3) to node {} (B2);

\draw[->] (A0) to node {$1$} (B0);

\node(C1)[right of=B0] {$\proj_{i}$};
\node(C2)[right of=C1] {$\cdots$};
\node(C3)[right of=C2] {$\proj_{i+1}$};
\node(C4)[right of=C3] {$\proj_{i}$};

\draw[->] (B0) to node {} (C1);
\draw[->] (C1) to node {} (C2);
\draw[->] (C2) to node {} (C3);
\draw[->] (C3) to node {} (C4);

\end{tikzpicture}$$

Below is the map $h_{\simp_i \proj_{i}}^{(n)} \circ h_{\proj_{i} \simp_i}^{(k)}$.
$$\begin{tikzpicture}[scale=10, node distance=1.5cm]
\node(A0) {$\proj_{i}$};
\node(A1)[left of=A0] {$\proj_{i+1}$};
\node(A2)[left of=A1] {$\proj_{i}$};
\node(A3)[left of=A2] {$\cdots$};
\node(A4)[right of=A0] {$\cdots$};
\node(A5)[right of=A4] {$\proj_{i}$};
\draw[->] (A1) to node {} (A0);
\draw[->] (A2) to node {} (A1);
\draw[->] (A3) to node {} (A2);
\draw[->] (A0) to node {} (A4);
\draw[->] (A4) to node {} (A5);

\node(B0)[below of=A0] {$\proj_{i}$};
\node(B1)[below of=A1] {$\proj_{i+1}$};
\node(B2)[below of=A2] {$\proj_{i}$};
\node(B3)[below of=A3] {$\cdots$};

\draw[->] (B1) to node {} (B0);
\draw[->] (B2) to node {} (B1);
\draw[->] (B3) to node {} (B2);

\draw[->] (A0) to node {$1$} (B0);

\node(C1)[right of=B0] {$\proj_{i+1}$};
\node(C2)[right of=C1] {$\cdots$};
\node(C3)[right of=C2] {$\proj_{i+1}$};
\node(C4)[right of=C3] {$\proj_{i}$};

\draw[->] (B0) to node {} (C1);
\draw[->] (C1) to node {} (C2);
\draw[->] (C2) to node {} (C3);
\draw[->] (C3) to node {} (C4);

\end{tikzpicture}$$

Below is the map $p_i \circ h_{\proj_{i} \simp_i}^{(k)}$.
$$\begin{tikzpicture}[scale=10, node distance=1.5cm]
\node(A0) {$\proj_{i}$};
\node(A1)[left of=A0] {$\proj_{i+1}$};
\node(A2)[left of=A1] {$\proj_{i}$};
\node(A3)[left of=A2] {$\cdots$};
\node(A4)[right of=A0] {$\cdots$};
\node(A5)[right of=A4] {$\proj_{i}$};
\draw[->] (A1) to node {} (A0);
\draw[->] (A2) to node {} (A1);
\draw[->] (A3) to node {} (A2);
\draw[->] (A0) to node {} (A4);
\draw[->] (A4) to node {} (A5);

\node(B0)[below of=A0] {$\proj_{i}$};
\node(B1)[below of=A1] {$\proj_{i+1}$};
\node(B2)[below of=A2] {$\proj_{i}$};
\node(B3)[below of=A3] {$\cdots$};

\draw[->] (B1) to node {} (B0);
\draw[->] (B2) to node {} (B1);
\draw[->] (B3) to node {} (B2);

\draw[->] (A0) to node {$1$} (B0);



\end{tikzpicture}$$

Below is the map $h_{\simp_{i+1} \proj_{i+1}}^{(n)} \circ h_{\proj_{i+1} \simp_i}^{(k)}$.
$$\begin{tikzpicture}[scale=10, node distance=1.5cm]
\node(A0) {$\proj_{i+1}$};
\node(A1)[left of=A0] {$\proj_{i}$};
\node(A2)[left of=A1] {$\proj_{i+1}$};
\node(A3)[left of=A2] {$\cdots$};
\node(A4)[right of=A0] {$\cdots$};
\node(A5)[right of=A4] {$\proj_{i}$};
\draw[->] (A1) to node {} (A0);
\draw[->] (A2) to node {} (A1);
\draw[->] (A3) to node {} (A2);
\draw[->] (A0) to node {} (A4);
\draw[->] (A4) to node {} (A5);

\node(B0)[below of=A0] {$\proj_{i+1}$};
\node(B1)[below of=A1] {$\proj_{i}$};
\node(B2)[below of=A2] {$\proj_{i+1}$};
\node(B3)[below of=A3] {$\cdots$};

\draw[->] (B1) to node {} (B0);
\draw[->] (B2) to node {} (B1);
\draw[->] (B3) to node {} (B2);

\draw[->] (A0) to node {$1$} (B0);

\node(C1)[right of=B0] {$\proj_{i}$};
\node(C2)[right of=C1] {$\cdots$};
\node(C3)[right of=C2] {$\proj_{i}$};
\node(C4)[right of=C3] {$\proj_{i+1}$};

\draw[->] (B0) to node {} (C1);
\draw[->] (C1) to node {} (C2);
\draw[->] (C2) to node {} (C3);
\draw[->] (C3) to node {} (C4);

\end{tikzpicture}$$

Below is the map $p_{i+1} \circ h_{\proj_{i+1} \simp_i}^{(k)}$.
$$\begin{tikzpicture}[scale=10, node distance=1.5cm]
\node(A0) {$\proj_{i+1}$};
\node(A1)[left of=A0] {$\proj_{i}$};
\node(A2)[left of=A1] {$\proj_{i+1}$};
\node(A3)[left of=A2] {$\cdots$};
\node(A4)[right of=A0] {$\cdots$};
\node(A5)[right of=A4] {$\proj_{i}$};
\draw[->] (A1) to node {} (A0);
\draw[->] (A2) to node {} (A1);
\draw[->] (A3) to node {} (A2);
\draw[->] (A0) to node {} (A4);
\draw[->] (A4) to node {} (A5);

\node(B0)[below of=A0] {$\proj_{i+1}$};
\node(B1)[below of=A1] {$\proj_{i}$};
\node(B2)[below of=A2] {$\proj_{i+1}$};
\node(B3)[below of=A3] {$\cdots$};

\draw[->] (B1) to node {} (B0);
\draw[->] (B2) to node {} (B1);
\draw[->] (B3) to node {} (B2);

\draw[->] (A0) to node {$1$} (B0);



\end{tikzpicture}$$

Below is the map $h_{\simp_i \proj_{i+1}}^{(n)} \circ h_{\proj_{i+1} \simp_{i+1}}^{(k)}$.
$$\begin{tikzpicture}[scale=10, node distance=1.5cm]
\node(A0) {$\proj_{i+1}$};
\node(A1)[left of=A0] {$\proj_{i}$};
\node(A2)[left of=A1] {$\proj_{i+1}$};
\node(A3)[left of=A2] {$\cdots$};
\node(A4)[right of=A0] {$\cdots$};
\node(A5)[right of=A4] {$\proj_{i+1}$};
\draw[->] (A1) to node {} (A0);
\draw[->] (A2) to node {} (A1);
\draw[->] (A3) to node {} (A2);
\draw[->] (A0) to node {} (A4);
\draw[->] (A4) to node {} (A5);

\node(B0)[below of=A0] {$\proj_{i+1}$};
\node(B1)[below of=A1] {$\proj_{i}$};
\node(B2)[below of=A2] {$\proj_{i+1}$};
\node(B3)[below of=A3] {$\cdots$};

\draw[->] (B1) to node {} (B0);
\draw[->] (B2) to node {} (B1);
\draw[->] (B3) to node {} (B2);

\draw[->] (A0) to node {$1$} (B0);

\node(C1)[right of=B0] {$\proj_{i+1}$};
\node(C2)[right of=C1] {$\cdots$};
\node(C3)[right of=C2] {$\proj_{i+1}$};
\node(C4)[right of=C3] {$\proj_{i}$};

\draw[->] (B0) to node {} (C1);
\draw[->] (C1) to node {} (C2);
\draw[->] (C2) to node {} (C3);
\draw[->] (C3) to node {} (C4);

\end{tikzpicture}$$

\begin{theorem}\label{bigcomputationtheorem}
All of the non-trivial operations for $\PA$ are given below
along with their counterparts from the symmetry $\kappa$ coming from exchanging the nodes $1$ and $2$ in the underlying quiver.  

This list contains the $m_2$ operations.
$$\begin{array}{llcll}
m_2(j_1,p_2) = (21) 
\\
m_2(u_1^n,u_1^m) = u_1^{n+m} 
\\
m_2(\b u_1^m,u_1^n) = \b u_1^{n+m} 
\\
m_2(u_2^m, \b u_1^n) = \b u_1^{n+m} 
\\
m_2(\a u_2^m, \b u_1^n) = u_1^{n+m+1} 
\end{array}$$

This list contains the higher operations.
$$\begin{array}{lcl}
m_{2n+2k+5}((12),(212)^k,j_1,\b u_1^{n+k+1},p_1,(212)^n, (21)) =1_{\proj_2}  \\
m_{2n+3}((12),(212)^n,j_1, \b u_1^n)=j_2 \\
m_{2n+3}((212)^k, j_1, \b u_1^n, p_1, (212)^{n-k})=1_{\proj_1} \\
m_{2n+3}(u_1^k, p_1, (212)^{n},j_1)= u_1^{k-n-1} \a \\
m_{2n+3}(u_2^k \b, p_1, (212)^n,j_1) = u_2^{k-n} \\
m_{2n+4}(u_1^k, p_1, (212)^n, (21), j_2) = u_1^{k-n-1} \\
m_{2n+4}(p_1, (212)^n, (21), j_2, u_1^k) = u_1^{k-n-1} \\
m_{2n+4}(u_2^k \b, p_1, (212)^n, (21), j_2) = u_2^{k-n-1} \b \\
m_{2n+3}(u_2^n \b, p_1, (212)^n, (21)) = p_2 \\
m_{2n+2}(u_1^n, p_1, (212)^n)=p_1 \\
m_{2k+2n+4}((121)^k, j_2, u_1^{n+k+1}, p_1, (212)^n, (21))=1_{\proj_2} \\
m_{2n+4}((121)^{n+1}, j_2, u_1^{n+1})=j_2 \\
m_{2n+4}((212)^k, (21), j_2, u_1^{n+1}, p_1, (212)^{n-k})=1_{\proj_1} \\
m_{2m+3}(p_2,(121)^m,j_2,\a u_2^n) = u_2^{n-m} \\
m_{2m+3}(p_2,(121)^m,j_2, u_1^n) = \b u_1^{n-(m+1)} \\
 m_{2m+4}(p_1,(21),(121)^m,j_2,\a u_2^{n+1}) = \a u_2^{n-m}.
\end{array}$$

\end{theorem}

\begin{proof}
The $\Ainf$-structure on $\Pi$ is determined by the dg structure on $\End^*(\tM)$ and  Theorem ~\ref{HPthm}. The dg category $\End^*(\tM)$ only has only non-trivial first and second multiplications (the derivation and the natural algebra multiplication).  See Example \ref{endex} for more details.

Since all of the higher multiplications in $\End^*(\tM)$ are trivial,
in order to compute $m_n$ we must determine all possible binary trees with $n$ input edges satisfying certain properties.
Let $f_n, \ldots, f_1 \in \Pi$ such that the composition $f_{i+1} \circ f_i$ makes sense for $i=1,\ldots,n-1$.
The input edges (read from left to right) are labeled $f_n, \ldots, f_1$.
First one includes each $f_i \in \Pi$ into $\End^*(\tM)$.
The internal edges are labeled by $H$.

From a calculation we see that $H(f_{i+1} \circ f_i)$ for $f_{i+1}, f_i \in \Pi$ is non-zero only in the cases listed in ~\eqref{initialhs}.

From a calculation we see that higher homotopy maps are produced only when $H$ is applied to a product of the form $hf$ or $fh$ where $h$ is some higher homotopy map and $f$ is an element in $\Pi = \Ext^*(M)$.  The possibilities are listed in ~\eqref{higherhs}.
In particular, the only non-zero way to grow a binary tree labeled as in \S \ref{HPTsec} is illustrated in the remark below.

Finally, to produce an element in $\Pi$, one must apply the projection map $p$ described in \eqref{perturbdiag} to certain products of two elements of $\End^*(\tM)$
listed in ~\eqref{hstoext}.

Any other composition that we need to consider produces a non-cycle and so $H$ of it is set to zero. These cases are enumerated above as partial terminating operations.
\end{proof}

\begin{rmk}\label{rmkev2}
  The $\Ainf$-maps $m_2$ are determined by the composition in the $Ext$-algebra. The maps $m_3$ and $m_4$ can be done by hand. The only way to inductively evolve a binary tree to give non-zero higher operation is illustrated below, see Rmk. \ref{rmkev}.
\begin{tikzpicture}[scale=10, node distance=.68cm]
  \node (0A1) {$p_1$};
  \node (0A2) [right of=0A1] {};
  \node (0A3) [right of=0A2] {$(21)$};
  \node (0A4) [right of=0A3] {$(12)$};
  \node (0B2) [below of=0A2] {$\mu$};
  \node (0B3) [below of=0A3] {};
  \node (0C3) [below of=0B3] {$\mu$};
\node (0D3) [below of=0C3] {};
  \draw[] (0A1) to node {} (0B2);
  \draw[] (0A3) to node {} (0B2);
  \draw[] (0B2) to node [swap] {$H$} (0C3);
  \draw[] (0A4) to node {} (0C3);
\draw[] (0C3) to node {} (0D3);
\node (TMP1) [below=.75 of 0D3] {};
  \node (1A1) [left=2 of TMP1] {$p_1$};
  \node (1A2) [right of=1A1] {};
  \node (1A3) [right of=1A2] {$(21)$};
  \node (1A4) [right of=1A3] {$(12)$};
\node (1A5) [right of=1A4] {$(21)$};
  \node (1B2) [below of=1A2] {$\mu$};
  \node (1B3) [below of=1A3] {};
  \node (1C3) [below of=1B3] {$\mu$};
  \node (1B4) [below of=1A4] {};
  \node (1C4) [below of=1B4] {};
\node (1D4) [below of=1C4] {$\mu$};
\node (1E4) [below of=1D4] {};
  \draw[] (1A1) to node {} (1B2);
  \draw[] (1A3) to node {} (1B2);
  \draw[] (1B2) to node [swap] {$H$} (1C3);
  \draw[] (1A4) to node {} (1C3);
\draw[] (1C3) to node [swap] {$H$} (1D4);
\draw[] (1D4) to node {} (1E4);
\draw[] (1A5) to node {} (1D4);
  \node (2A1) [right=2 of TMP1] {$p_1$};
  \node (2A2) [right of=2A1] {};
  \node (2A3) [right of=2A2] {$(21)$};
  \node (2A4) [right of=2A3] {$(12)$};
\node (2A5) [right of=2A4] {$j_1$};
  \node (2B2) [below of=2A2] {$\mu$};
  \node (2B3) [below of=2A3] {};
  \node (2C3) [below of=2B3] {$\mu$};
  \node (2B4) [below of=2A4] {};
  \node (2C4) [below of=2B4] {};
\node (2D4) [below of=2C4] {$\mu$};
\node (2E4) [below of=2D4] {};
  \draw[] (2A1) to node {} (2B2);
  \draw[] (2A3) to node {} (2B2);
  \draw[] (2B2) to node [swap] {$H$} (2C3);
  \draw[] (2A4) to node {} (2C3);
\draw[] (2C3) to node [swap] {$H$} (2D4);
\draw[] (2D4) to node {} (2E4);
\draw[] (2A5) to node {} (2D4);
\draw[->,dashed,bend left=5] (0D3) to node {} (1A3);
\draw[->,dashed,bend right=5] (0D3) to node {} (2A3);

\node (TMP2) [below=.75 of 1E4] {};
  \node (3A1) [left=2.5 of TMP2] {$p_1$};
  \node (3A2) [right of=3A1] {};
  \node (3A3) [right of=3A2] {$(21)$};
  \node (3A4) [right of=3A3] {$(12)$};
\node (3A5) [right of=3A4] {$(21)$};
\node (3A6) [right of=3A5] {$(12)$};
  \node (3B2) [below of=3A2] {$\mu$};
  \node (3B3) [below of=3A3] {};
  \node (3C3) [below of=3B3] {$\mu$};
  \node (3B4) [below of=3A4] {};
  \node (3C4) [below of=3B4] {};
\node (3D4) [below of=3C4] {$\mu$};
\node (3E4) [below of=3D4] {};
\node (3B5) [below of=3A5] {};
\node (3C5) [below of=3B5] {};
\node (3D5) [below of=3C5] {};
\node (3E5) [below of=3D5] {$\mu$};
\node (3F5) [below of=3E5] {};
  \draw[] (3A1) to node {} (3B2);
  \draw[] (3A3) to node {} (3B2);
  \draw[] (3B2) to node [swap] {$H$} (3C3);
  \draw[] (3A4) to node {} (3C3);
\draw[] (3C3) to node [swap] {$H$} (3D4);
\draw[] (3D4) to node [swap] {$H$} (3E5);
\draw[] (3A5) to node {} (3D4);
\draw[] (3E5) to node {} (3F5);
\draw[] (3A6) to node {} (3E5);
  \node (4A1) [right=2.5 of TMP2] {$p_1$};
  \node (4A2) [right of=4A1] {};
  \node (4A3) [right of=4A2] {$(21)$};
  \node (4A4) [right of=4A3] {$(12)$};
\node (4A5) [right of=4A4] {$(21)$};
\node (4A6) [right of=4A5] {$j_2$};
  \node (4B2) [below of=4A2] {$\mu$};
  \node (4B3) [below of=4A3] {};
  \node (4C3) [below of=4B3] {$\mu$};
  \node (4B4) [below of=4A4] {};
  \node (4C4) [below of=4B4] {};
\node (4D4) [below of=4C4] {$\mu$};
\node (4E4) [below of=4D4] {};
\node (4B5) [below of=4A5] {};
\node (4C5) [below of=4B5] {};
\node (4D5) [below of=4C5] {};
\node (4E5) [below of=4D5] {$\mu$};
\node (4F5) [below of=4E5] {};
  \draw[] (4A1) to node {} (4B2);
  \draw[] (4A3) to node {} (4B2);
  \draw[] (4B2) to node [swap] {$H$} (4C3);
  \draw[] (4A4) to node {} (4C3);
\draw[] (4C3) to node [swap] {$H$} (4D4);
\draw[] (4D4) to node [swap] {$H$} (4E5);
\draw[] (4A5) to node {} (4D4);
\draw[] (4E5) to node {} (4F5);
\draw[] (4A6) to node {} (4E5);

\draw[->,dashed,bend left=5] (1E4) to node {} (3A4);
\draw[->,dashed,bend right=5] (1E4) to node {} (4A3);

\node (TMP3) [below of=2E4] {};
\node (TMP4) [right=2 of TMP3] {$0$};
\draw[->,dashed,bend right=5] (2E4) to node {} (TMP4);

\node (TMP5) [below of=4F5] {};
\node (TMP6) [right=2 of TMP5] {$0$};
\draw[->,dashed,bend right=5] (4F5) to node {} (TMP6);

\node (TMP7) [below of=3F5] {};
\node (TMP8) [right=2 of TMP7] {$\ldots$};
\node (TMP9) [left=2 of TMP7] {$\ldots$};
\draw[->,dashed,bend right=5] (3F5) to node {} (TMP8);
\draw[->,dashed,bend left=5] (3F5) to node {} (TMP9);
\end{tikzpicture}

\end{rmk}

\subsection{The preprojective cone}\label{lastsecsec}
Recall from Def. \ref{completedef} that our abstract cones are determined by
a certain lifting problem. This section combines all of the bits and pieces
from previous sections and provides some explanation as to what we have
computed.

Before proceeding, it is useful to observe the following remark.
\begin{rmk}
  If $F : \eC\to \eD$ is an $\Ainf$-functor such that $F^d = 0$ for $d\geq 2$, then the $\Ainf$-relation in Def. \ref{functordef} becomes
  \begin{equation}\label{simpleainf} 
m^\eD_d(F^1(f_d),\cdots, F^1(f_1)) = F^1(m_d^{\eC}(f_d,\cdots, f_1)).
  \end{equation}
\end{rmk}

The proposition below is a detailed version of the comments at the end of Ex. \ref{ibcompex} in the preprojective setting.

\begin{prop}\label{embprop}
The subcategory $\pi$ associated to the two simple modules $S_1, S_2\in \Pi A_2\module$ 
$$\begin{tikzpicture}[scale=10, node distance=3cm]
\node (S) {$S_1$};
\node (A) [right of=S] {$S_2$};
\draw[<-,bend left=32] (S) to node {$\a$} (A);
\draw[<-,bend left=32] (A) to node {$\b$} (S);
\end{tikzpicture}$$ 
is formal; the higher $\Ainf$-structure $m_d^\pi = 0$ for $d > 2$. The inclusion $\iota : \pi \hookrightarrow \Pi$ is an $\Ainf$-functor $\iota = \{ \iota^d \}$ which is determined by the assignments:
$\iota(S_i) := S_i$ on objects, $\iota^1(\a) := \a$, $\iota^1(\b) := \b$ on maps and $\iota^d := 0$ for $d\geq 2$.
\end{prop}
\begin{proof}
The category $\pi$ is formal by Theorem \ref{bigcomputationtheorem} since there
are no non-trivial higher operations in the list involving only entries of
the form $\a$ and $\b$.  Since $\pi$ is formal, we need only check that
$\iota$ satisfies Eqn. \eqref{simpleainf}. This again follows from the
observation that there are no relations among $\a$ and $\b$ in $\Pi$ and
there are no higher $\Ainf$-relations, $m^\Pi_d\vert_{\pi} \equiv 0$ for
$d > 2$, by Theorem \ref{bigcomputationtheorem}.
\end{proof}

In the notation introduced by the proposition, the lifting problem in Def. \ref{completedef} can be restated by the commutative diagram below.
$$\begin{tikzpicture}[scale=10, node distance=1.5cm]
\node (A1) {$\pi$};
\node (C1) [below of=A1] {$\Pi$};
\node (D1) [right of=A1] {$\eD$};
\draw[right hook->] (A1) to node [swap] {$\iota$} (C1);
\draw[->] (A1) to node {$F$} (D1);
\draw[dashed,->] (C1) to node [swap] {$\tilde{F}$} (D1);
\end{tikzpicture}$$
So the initial data is an $\Ainf$-functor $F : \pi\to\eD$ and a cone on $F$ is determined by a lift along $\iota$, i.e. an $\Ainf$-functor $\tilde{F} : \Pi \to \eD$ for which $\tilde{F}\circ \iota \simeq F$ in the category $\AAinf(\pi, \eD)$.

The theorem below shows that the upper and lower parts of the $\Pi$-diagram
at the beginning of \S \ref{picatsec} are triangles in the sense of
Theorem \ref{deltathm}. It follows that, up to homotopy, the portions of the
category in the completion, $\tilde{F}(P_1)$ and $\tilde{F}(P_2)$, are classical  cones on the maps $F^1(\b)$ and $F^1(\a)$ respectively.

  \begin{thm}\label{trithmthm}
    If $F : \Pi \to \eD$ is an $\Ainf$-functor from the preprojective category to an $\Ainf$-category $\eD$ then the objects
    associated to $P_i$ in $\eD$ are homotopy equivalent to cones on the maps $\a$ and $\b$. More precisely,
\begin{equation}\label{alabeleqn}
  F(P_1) \simeq C(F^1(\b))\conj{ and } F(P_2) \simeq C(F^1(\a)).
\end{equation}  
    \end{thm}
    \begin{proof}
The category $\Delta$ is pictured in Def. \ref{tricatdef}. There are three objects $A$,$B$ and $C$ together with maps $\a : A \to B$, $\b : B\to C$ and $\ga : C\to A$ with degrees $\vnp{\a} = 0$, $\vnp{\b} = 0$ and $\vnp{\ga} = 1$.

For $i=1,2$, there are assignments $\iota_i : \Ob(\Delta) \hookrightarrow \Ob(\Pi)$ and  $\iota_i^1 : \Hom_{\Delta}(X_1,X_0) \to Hom_{\Pi}(\iota_i(X_1),\iota_i(X_0))$ given by
$$      \begin{array}{l||c|c|c|}
        & A   & B & C\\
\hline\hline
\iota_1 & S_2 & P_1 & S_1 \\
\hline
\iota_2 & S_1 & P_2 & S_2 \\
\hline
        \end{array}
\conj{ and }
      \begin{array}{l||c|c|c|}
        & \a   & \b & \ga\\
\hline\hline
\iota_1^1 & j_1 & p_1 & \b \\
\hline
\iota_2^1 & j_2 & p_2 & \a \\
\hline
        \end{array}.
$$
Set $\iota_i^d := 0$ for $d \geq 2$. In order to see that the assignments $\iota_i$ are $\Ainf$-functors, we must check that the $\Ainf$-relation Eqn. \eqref{simpleainf} above is satisfied. The cases $\iota_1$ and $\iota_2$ are symmetric. So consider $\iota_1$, then Eqn. \eqref{simpleainf} holds because the only non-identity compositions are zero
$$j_1 \circ \b = 0 \conj{} p_1\circ j_1 = 0 \conj{} \b \circ p_1 = 0$$
and the only higher $\Ainf$-operation supported by the morphisms $\{ p_1, j_1, \b \}$ in $\Pi$ are given by the $m_3$ in Eqn. \eqref{basictri1eqn}; these agree with Eqn. \eqref{m3def}.

Thus the restrictions formed by the compositions $F\circ \iota_i : \Delta \to \eD$ are $\Ainf$-functors. By Theorem \ref{deltathm}, 
each of the two restrictions determines a triangle in $\eD$. In particular, the image of each vertex of $\Delta \subset \Pi$ must be homotopic to the cone on the morphism in subcategory complementary to the vertex.  A special case of this is Eqn. \eqref{alabeleqn}.
\end{proof}

\begin{example}
Suppose we have an $\Ainf$-functor $F : \pi \to \eD$. The structure maps $F : \Ob(\pi) \to \Ob(\eD)$ and $F^1_{X,Y} : \Hom_\pi(X,Y) \to \Hom_{\eD}(F(X),F(Y))$ determine two cycles 
$$\begin{tikzpicture}[scale=10, node distance=3cm]
\node (S) {$S_1'$};
\node (A) [right of=S] {$S_2'$};
\draw[<-,bend left=32] (S) to node {$\a'$} (A);
\draw[<-,bend left=32] (A) to node {$\b'$} (S);
\end{tikzpicture}$$ 
where $S_1' = F(S_1)$, $S_2'= F(S_2)$, $\a' = F^1(\a)$ and $\b' = F^1(\b)$. If $\eD$ is triangulated then there are objects $C(\a')$ and $C(\b')$ in $\eD$ which are cones in the sense of Def. \ref{inconedef}. By the theorem above, a lift $\tilde{F} : \Pi \to \eD$ of $F$ along $\iota$ must associate to $P_1$ and $P_2$ objects which are homotopy equivalent to $C(\b')$ and $C(\a')$ respectively. When $\eD$ is a dg category of complexes, this can be made very explicit. The diagram for $\Pi$ at the beginning of \S \ref{picatsec} becomes the one below.
$$\begin{tikzpicture}[scale=10, node distance=2cm]
\node (S_1) {$S_1'$};
\node (A) [right of=S_1] {};
\node (B) [right of=A] {};
\node (C) [right of=B] {};
\node (S_2) [right of=C] {$S_2'$};
\node (A_2) [above of=B] {$(S_1'\oplus S_2[1]',d_{C(\a')})$};
\node (A_1) [below of=B] {$(S_1[1]'\oplus S_2',d_{C(\b')})$};
\draw[->] (S_1) to node {$(1,0)^t$} (A_2);
\draw[->] (A_2) to node {$(0,1)$} (S_2);
\draw[->] (S_2) to node {$(0, 1)^t$} (A_1);
\draw[->] (A_1) to node {$(1,0)$} (S_1);
\draw[transform canvas={yshift=+0.5ex},->] (S_2) to node[pos=.75,above] {$\a'$} (S_1);
\draw[transform canvas={yshift=-0.5ex},->] (S_1) to node[pos=.75,below] {$\b'$} (S_2);
\end{tikzpicture}$$
Since $(21) = j_1\circ p_2$ and $(12) = j_2\circ p_1$, the matrices associated to these maps are
$$(21) = \left(\begin{array}{cc} 1 & 0 \\ 0 & 0 \end{array}\right) \conj{ and } (12) = \left(\begin{array}{cc} 0 & 0 \\ 0 & 1 \end{array}\right)$$
and we see directly the relations $(12)(21) = 0$ and $(21)(12) = 0$.
\end{example}

\subsection{Coda}\label{codasec}

The remainder of this paper establishes some context which pertains to
future work.

\begin{rmk}
 Just as Proposition \ref{embprop} shows that the subcategory $\pi$ determined by simples embeds in $\Pi$, the quiver presentation for $\Pi A_2$ in Def. \ref{pai2def} embeds in $\Pi$ by the Yoneda embedding. So for $\Pi$, there are two notions of map which yield the same notion of an $\iota$-distinguished triangle.
One comes from the $\a$ and $\b$ maps between simple modules, and another comes from the $(21)$ and $(12)$ maps between projective modules. 
These two are dual in the sense that some of the data arising from an $\iota$-distinguished triangle of one $\iota$-cone construction is the same as the initial data for the other.
Since exchanging simples and projectives results only in a rotation of the triangle in the construction of the $A_2$-cone, this shows that the preprojective cones exhibit some new behavior.
\end{rmk}

\begin{rmk}\label{postnikovrmk}
The simplicity of Theorem \ref{trithmthm} seems
incongruous with the complexity of the $\Ainf$-structure found in Theorem
\ref{bigcomputationtheorem}. The $\Ainf$-structure is complicated in part
because of the trick in Eqn. \eqref{trickeqn} and in part because it is
recording higher Postnikov systems among compositions of $\a$ and $\b$
maps. We will give an informal explanation of the first
non-trivial example. 

Recall from \S \ref{reltrisec} that functors $\Delta \to \eD$ correspond to
distinguished triangles in $\eD$ and that $\Delta$ agrees with the partially
wrapped Fukaya category $\aF_3$ of the disk with three marked points.
There is an extension of these statements.

\begin{defn}\label{partfukdef}
  The partially wrapped Fukaya category $\aF_n$ of the disk $D^2$ with $n$ marked points along the boundary is the $\Ainf$-category with $n$-objects $\Ob(\aF_n) = \{X_i\}_{i\in \ZZ/n}$ and maps $1_{X_i} : X_i \to X_i$ and $f_i : X_i \to X_{i+1}$. The gradings are chosen to satisfy the constraint $\sum_{i} \vnp{f_i} = n-2$. The only non-trivial $\Ainf$-operations are compositions with identity and cyclic permutations of
  $$m_n(f_n,\ldots, f_1) = 1_{X_1}.$$
\end{defn}

See \cite[above Eqn. (3.21)]{HKK}, \cite{Nadler, DK} or other references \cite[I (3g) Rmk. 3.11]{Seidel}. For $n\geq 3$, functors from the partially wrapped Fukaya category $\aF_n \to \eD$ correspond to $n$-fold extensions among objects in $\eD$. Since there is a Morita equivalence
$$\aF_n \sqcup_{i,j} \aF_m \xto{\sim} \aF_{n+m-2}$$
associated to the gluing of the $i$th boundary object in $(D^2,n)$ to the $j$th boundary object in $(D^2,m)$ \cite[\S 3.6]{HKK}, functors $\aF_n \to \eD$ correspond to $n$-fold extensions among objects in $\eD$ since every disk with $n$ marked points can be subdivided into a gluing of disks with $3$ marked points.
On the other hand, precisely the same logic as was used in Theorem \ref{trithmthm} to establish the existence of functors (for $i=1,2$) $\iota_i : \Delta \to \Pi$ for the triangles formed by the maps $\{j_1,p_1,\b\}$ and $\{j_2,p_2,\a\}$ corresponding to $\Ainf$-operations in Eqn. \eqref{basictri1eqn} can be used to establish the existence of $\Ainf$-functors
$\kappa_i : \aF_4 \to \Pi$
for the quadrilaterals corresponding to the maps $\{ p_1, (21), j_2, u_1 \}$ and $\{ p_2, (12), j_1, u_2 \}$ appearing in Eqn. \eqref{basictri2eqn}. So in addition to determining two distinguished triangles
$$ \cdots \to S_1' \xto{\b'} S_2' \to C(\b') \to S_1'[1]\to\cdots \conj{ and } \cdots \to S_2' \xto{\a'} S_1' \to C(\a') \to S_2'[1]\to \cdots$$
as in Thm. \ref{trithmthm} earlier, an $\Ainf$-functor $\tilde{F} : \Pi \to \eD$ corresponding to a cone on $F : \pi \to \eD$ determines 4-fold Postnikov systems via restrictions $\tilde{F}\kappa_i : \aF_4 \to \eD$.  For the maps $\{ p_2, (12), j_1, u_2 \}$ this corresponds to
$$\cdots \to C(\b') \xto{(12)'} C(\a') \xto{p_2'} S_2' \xto{u_2'} S_2'[2] \to \cdots$$
where $(12)'=\tilde{F}^1(12)$, $p_2'=\tilde{F}^1(p_2)$ and $u_2'=\tilde{F}^1(u_2)$.
Topologically the quadrilateral formed by the gluing of the two triangles below along the object $S_1'$.
\begin{center}
\begin{overpic}[scale=0.6]
{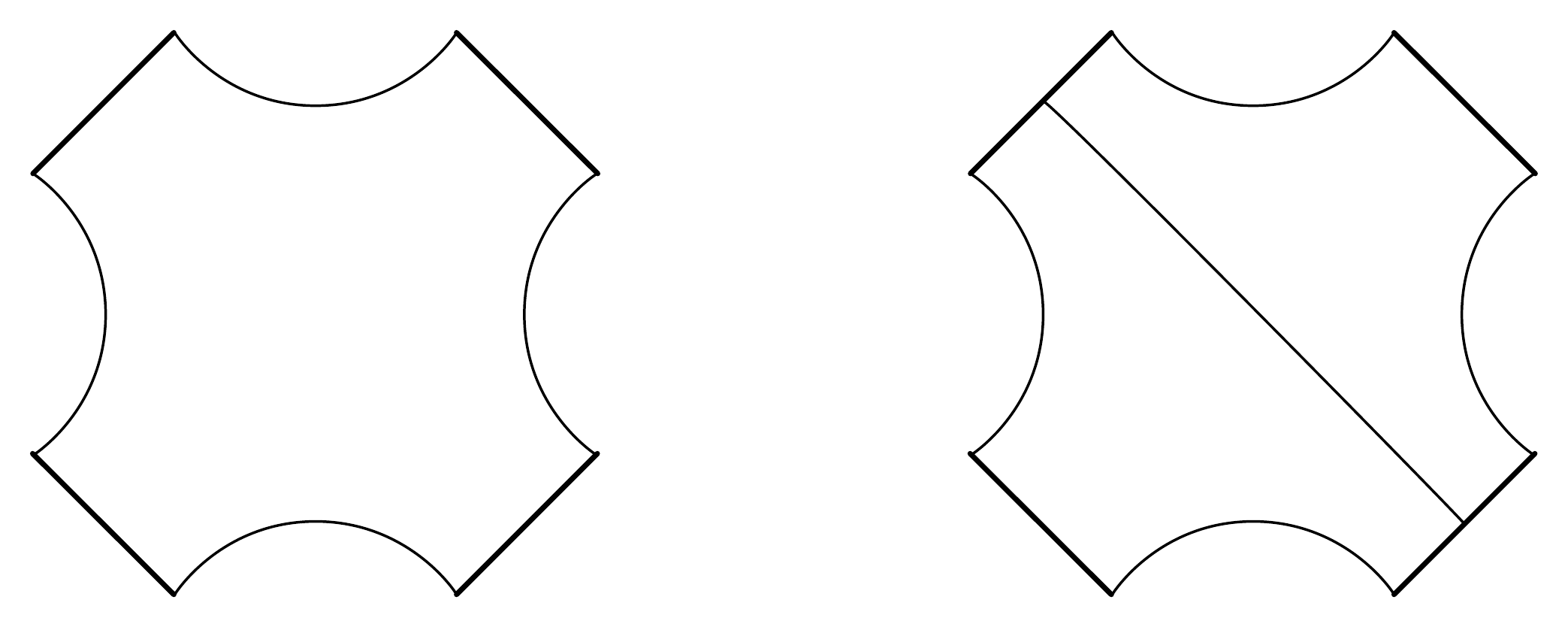}
\put(172,71){$\rightsquigarrow$}
\put(5, 128){$(12)'$}
\put(10,12){$j_1'$}
\put(125, 12){$u_2'$}
\put(125, 125){$p_2'$}
\put(60, 130){$C(\a')$}
\put(65, 10){$S_2'$}
\put(128, 70){$S_2'$}
\put(-9, 70){$C(\b')$}
\put(340,125){$p_2'$}
\put(225,12){$j_1'$}
\put(294, 79){$S_1'$}
\put(275, 130){$C(\a')$}
\put(280, 10){$S_2'$}
\put(343, 70){$S_2'$}
\put(206, 70){$C(\b')$}
\put(222,120){$p_1'$}
\put(237,133){$j_2'$}
\put(351,24){$\a'$}
\put(333,7){$\b'$}
\end{overpic}
\end{center}
\end{rmk}

We conclude with a connection between the $\Ainf$-category $\Pi$ and an important category coming from symplectic topology.
The wrapped Fukaya category $D^\pi\aW(P)$ of the pair of pants
$P=S^2 \backslash \{D^2_1, D^2_2, D^2_3\}$ can be generated by three
Lagrangians $X_0$, $X_1$ and $X_2$. The $\Ainf$-subcategory $\A$ determined
by these objects was studied in relation to the homological mirror symmetry
conjecture \cite{AAEKO}. Proposition \ref{aaekorelprop} constructs a functor
from this category to a version of the preprojective category $\Pi'$.

\begin{defn}{(\cite{AAEKO})}\label{aaekodef}
  The generating category $\A$ has objects $\Ob(\A) = \{ X_0, X_1, X_2\}$ and morphisms
  $$Hom_{\A}(X_i,X_j) = \left\{\begin{array}{ll} k[x_i,y_i]/(x_iy_i) & i = j \\ k[x_{i+1}]u_{i,i+1} = u_{i,i+1}k[y_i] & j = i+1 \\ k[y_{i-1}]v_{i,i-1} = v_{i,i-1}k[x_i] & j=i-1 \\ 0 & \normaltext{otherwise}. \end{array}\right.$$
These maps compose according to
\begin{equation}\label{compruleeqn}
  (x_i^k u_{i-1,i}) \circ (v_{i,i-1} x_i^l) = x_i^{k+l+1}  \conj{ and } (v_{i,i-1} x_i^l) \circ (x_i^k u_{i-1,i}) = y_{i-1}^{k+l+1}
\end{equation}  
and are graded by setting $\vnp{u_{0,1}} = 1$, $\vnp{v_{1,0}} = 1$ and $\vnp{u_{i,i+1}} = 0$, $\vnp{v_{i,i-1}} = 0$ in all other cases. This is pictured on the lefthand side of \eqref{2diagrams}.


The higher $\Ainf$-operations are determined by 
\begin{equation}\label{pantstri1eqn}
\begin{array}{lcl}
m_3(u_{2,0}, u_{1,2}, u_{0,1}) = 1_{X_0} & \quad & m_3(v_{1,0}, v_{2,1}, v_{0,2}) = 1_{X_0}\\
m_3(u_{0,1}, u_{2,0}, u_{1,2}) = 1_{X_1} & \quad & m_3(v_{2,1}, v_{0,2}, v_{1,0}) = 1_{X_1}\\
m_3(u_{1,2}, u_{0,1}, u_{2,0}) = 1_{X_2} & \quad & m_3(v_{0,2}, v_{1,0}, v_{2,1}) = 1_{X_2}\\
\end{array}
\end{equation}
\end{defn}

In order to make $\Pi$ look like $\A$ in Def. \ref{aaekodef}, we need to
combine the two projective objects $P_1$ and $P_2$ into one by forming the
direct sum $P := P_1\opp P_2$.  The following definition recalls how direct sums
are formed. This is a special case of  additivization \cite[I (3k)]{Seidel}.

\begin{defn} \label{adddef}
  If $X, Y\in \Ob(\aA)$ are objects in an $\Ainf$-category $\aA$ then the direct sum $X\opp Y\in D^\pi(\aA)$, \cite[I (4b)]{Seidel}, satisfies
\begin{align*}
  Hom(X\opp Y, Z) &= Hom(X,Z) \opp Hom(Y,Z), \conj{  } \\ Hom(W, X\opp Y) &= Hom(W,X) \opp Hom(W,Y),
  \end{align*}
  and the $\Ainf$-operations extend additively
$$m_k(a_k,\ldots, (a_i,a_i'), \ldots, a_1) = m_k(a_k,\ldots, a_i, \ldots, a_1) + m_k(a_k,\ldots, a_i', \ldots, a_1).$$
\end{defn}

This allows one to define an $\Ainf$-category $Mat(\mathcal{A})$ of direct sums of objects of $\mathcal{A}$.
The new category $\Pi'$ is formed by combining the two projectives.

\begin{defn}\label{piprimedef}
The category $\Pi'$ is the full $\Ainf$-subcategory of $Mat(\Pi)$ formed by the objects $\{S_1, S_2, P\}$ where $P:= P_1\opp P_2$. The $\Ainf$-structure of $\Pi'$ is determined by Thm. \ref{bigcomputationtheorem} and Def. \ref{adddef} above. The objects and morphisms in $\Pi'$ are pictured on the righthand side of \eqref{2diagrams}.


\end{defn}

\begin{equation}
\label{2diagrams}
\begin{tikzpicture}[scale=10, node distance=1.5cm]
\node (S) {$X_0$};
\node (X) [right=1.5cm of S] {};
\node (A) [right=1.5cm of X] {$X_1$};
\node (B) [below=2.598cm of X] {$X_2$};
\draw[<-,bend left=8] (S) to node {$v_{1,0}$} (A);
\draw[<-,bend left=8] (A) to node {$u_{0,1}$} (S);

\draw[<-,bend left=8] (B) to node {$u_{1,2}$} (A);
\draw[<-,bend left=8] (A) to node {$v_{2,1}$} (B);

\draw[<-,bend left=8] (B) to node {$v_{0,2}$} (S);
\draw[<-,bend left=8] (S) to node {$u_{2,0}$} (B);
\end{tikzpicture}
\conj{}
\begin{tikzpicture}[scale=10, node distance=1.5cm]
\node (S) {$S_1$};
\node (X) [right=1.5cm of S] {};
\node (A) [right=1.5cm of X] {$S_2$};
\node (B) [below=2.598cm of X] {$P$};
\draw[<-,bend left=8] (S) to node {$\a$} (A);
\draw[<-,bend left=8] (A) to node {$\b$} (S);

\draw[<-,bend left=8] (B) to node {$j_1$} (A);
\draw[<-,bend left=8] (A) to node {$p_2$} (B);

\draw[<-,bend left=8] (B) to node {$j_2$} (S);
\draw[<-,bend left=8] (S) to node {$p_1$} (B);
\end{tikzpicture}
\end{equation}

\begin{prop}\label{aaekorelprop}
There is a canonical $\Ainf$-functor $G : \A \to \Pi'$.
\end{prop}
\begin{proof}
  The $\Ainf$-functor $G = \{ G^d \}$ is defined by mapping the lefthand side of the diagram above to the righthand side of the diagram. In more detail, define $G : \Ob(\A)\to\Ob(\Pi')$ by setting $G(X_0) := S_1$, $G(X_1) := S_2$ and $G(X_2) := P$. The map $G^1 : \Hom_A(X_i,X_j) \to \Hom_{\Pi'}(G(X_i), G(X_j))$ is determined by setting
\begin{equation}\label{f1def}
\begin{array}{cc}
  G^1(u_{0,1})   := \b & G^1(v_{1,0}) := \a\\
    G^1(u_{1,2}) := j_1 & G^1(v_{2,1}) := p_2\\
    G^1(u_{2,0}) := p_1 &G^1(v_{0,2}) := j_2
\end{array}
\end{equation}
and the observation that Eqn. \eqref{compruleeqn} implies that the maps $u_{i,i+1}$ and $v_{i,i-1}$ for $i\in \ZZ/3$ generate $A$.

Since the higher maps $G^d := 0$ vanish for $d \geq 2$ then we need only check Eqn. \eqref{simpleainf}. The only non-trivial $\Ainf$-operations $m_n$ for $n>2$ in $A$ are determined by Eqn. \eqref{pantstri1eqn}, however Eqn. \eqref{f1def} can be used to translate back and forth between Eqn. \eqref{pantstri1eqn} and Eqn. \eqref{basictri1eqn}. This shows that $G$ satisfies Eqn. \eqref{simpleainf} for all of the $\Ainf$-operations $m_n$ for $n \geq 3$. 

When $n=2$, Eqn. \eqref{simpleainf} means that $G^1$ is a homomorphism. The only relation in $\A$ is $x_i y_i = 0$ in $End(X_i)$ for $i=0,1,2$. Eqn. \eqref{compruleeqn}, shows that
$$x_i = u_{i-1,i}v_{i,i-1} \conj{ and } y_i = v_{i+1,i} u_{i,i+1}$$
in $\A$. Combining this with Eqn. \eqref{f1def} gives
\begin{equation*}
  \begin{array}{ll}
x_0 = u_{2,0}v_{0,2} \mapsto p_1j_2 = 0 & y_0 = v_{1,0}u_{0,1} \mapsto \a\b = u_1\\
x_1 = u_{0,1}v_{1,0} \mapsto \b\a = u_2 & y_1 = v_{2,1}u_{1,2} \mapsto p_2j_1 = 0\\
x_2 = u_{1,2}v_{2,1} \mapsto j_1p_2 = (21) & y_2 = v_{0,2}u_{2,0} \mapsto j_2p_1 = (12).\\
  \end{array}  
  \end{equation*}
Finally, using this calculation we check that $x_iy_i = y_i x_i = 0$.
\begin{equation*}
  \begin{array}{ll}
x_0y_0 \mapsto 0 u_1 = 0 & y_0x_0 \mapsto u_1 0 = 0\\
x_1y_1 \mapsto u_2 0 = 0 & y_1x_1 \mapsto 0 u_2 = 0\\
x_2y_2 \mapsto (21)(12) = 0 & y_2x_2 \mapsto (12)(21) = 0.\\
  \end{array}  
  \end{equation*}
\end{proof}

\bibliographystyle{amsalpha}

\begin{thebibliography}{10}
\bibitem{Auroux} D. Auroux {\em A beginner's introduction to Fukaya categories}, Contact and symplectic topology, Bolyai Soc. Math. Stud., 26, J\'{a}nos Bolyai Math. Soc., Budapest, 85--136, 2014.
\bibitem{AAEKO} M. Abouzaid, D. Auroux, A. I. Efimov, L. Katzarkov and D. Orlov, {\em Homological mirror symmetry for punctured spheres}, J. Amer. Math. Soc., 26, 1051--1083, 2013.
\bibitem{BK} A. I. Bondal and M. M. Kapranov, {\em Enhanced Triangulated Categories}, Math. USSR. Sbornik, 70, 1, 93--107, 1991.
\bibitem{CB} W. Crawley-Boevey, {\em Geometry of the Moment Map for Representations of Quivers}, Compositio Math., 126, 257--293, 2001.
\bibitem{DK} T. Dyckerhoff, M. Kapranov, {\em Triangulated surfaces in triangulated categories}, arXiv:1306.2545.
\bibitem{Faonte} G. Faonte, {\em A-infinity functors and homotopy theory of dg-categories}, arXiv:1412.1255v1.
\bibitem{KLH} K. Lef\`{e}vre-Hasegawa, {\em Sur les A-infini cat\'{e}gories}, Univ. Paris 7 thesis, arXiv:math/0310337.
\bibitem{Leclercetal} C. Geiss, B. Leclerc and J. Schr\"{o}er, {\em Rigid modules over preprojective algebras}, Inventiones Mathematicae, 165, no 3, 589--632, 2006.
\bibitem{HKK} F. Haiden, L. Katzarkov and M. Kontsevich, {\em Flat surfaces and stability structures}, Publ. math. de l'IH\'{E}S, 126, 1, 247--318, 2017. 
\bibitem{Keller} B. Keller, {\em A-infinity algebras, modules and functor categories},
Homology, Homotopy and Applications, vol.3, no.1, 1--35, 2001.
\bibitem{KS} M. Khovanov and P. Seidel, {\em Quivers, Floer cohomology and braid group actions}, 	J. Amer. Math. Soc., 15, no. 1, 203--271, 2002. 
\bibitem{Madsen} D. Madsen, {\em Homological aspects in representation theory}, Thesis, Norwegian University of Science and Technology, 2002.
\bibitem{Markl} M. Markl, {\em Transferring $\Ainf$ (strongly homotopy associative) structures}, Rend. Circ. Mat. Palermo, 2, no. 79, 139--151, 2006.
\bibitem{Nadler} D. Nadler, {\em Cyclic symmetries of $A_n$-quiver representations}, Adv. in Math., 269, 346--363, 2015.
\bibitem{QS} Y. Qi and J. Sussan, {\em A Categorification of the Burau representation at prime roots of unity}, Selecta Mathematica, 22, 3, 1157--1193, 	2016. 
\bibitem{Ringel} C. M. Ringel, {\em The preprojective algebra of a quiver}, Algebras and modules II, CMS Conf. Proc., 24, Amer. Math. Soc., Providence, RI, 467--480, 1998.
\bibitem{Seidel} P. Seidel, {\em Fukaya categories and Picard-Lefschetz theory}, Zurich Lectures in Adv. Math., EMS, 2008.
\bibitem{Soibelman} Y. Soibelman, {\em Non-commutative geometry and deformations of $A_\infty$-algebras and $A_\infty$-categories}, Arbeitstagung 2003.
\bibitem{Toen} B. To\"{e}n, {\em The homotopy theory of dg-categories and derived Morita theory}, Invent. Math., 167, no. 3, 615--667, 2007.
\end{thebibliography}

\end{document}